\documentclass{siamart171218}

\usepackage{graphicx}
\usepackage{amsmath}
\usepackage{graphics}
\usepackage{amsfonts}
\usepackage{amssymb}%
\usepackage{amssymb}
\usepackage{mathrsfs}
\usepackage{amsfonts}
\usepackage{bbding}
\usepackage{amsmath}
\usepackage{bm}
\usepackage{graphicx}
\usepackage{multirow}
\usepackage{caption2}
\usepackage{float}
\usepackage{booktabs}
\usepackage{algorithm}
\usepackage{algorithmic}
\usepackage{subfigure}

\usepackage{array}

\usepackage{amsmath}
\usepackage{amsfonts}
\usepackage{graphicx}
\usepackage{epstopdf}
\usepackage{comment}
\usepackage{multirow}
\usepackage{subfigure}
\usepackage{algorithmic}
\usepackage{color}

\ifpdf
  \DeclareGraphicsExtensions{.eps,.pdf,.png,.jpg}
\else
  \DeclareGraphicsExtensions{.eps}
\fi

\newsiamremark{remark}{Remark}
\newsiamremark{hypothesis}{Hypothesis}
\crefname{hypothesis}{Hypothesis}{Hypotheses}
\newsiamthm{claim}{Claim}
\newsiamthm{assumption}{Assumption}

\headers{Integral Operator Approaches for Spherical Data}{S.-B. Lin}

\title{Integral Operator Approaches for Scattered Data Fitting   on Spheres \thanks{The corresponding author is Shao-Bo Lin. 
\funding{S. B. Lin  was partially
	supported by the  National Key R\&D Program of China (No.2020YFA0713900) and the
    Natural Science Foundation of China [Grant No  62276209].}}}

\author{Shao-Bo Lin\thanks{Center for Intelligent Decision-Making and Machine Learning, School of Management, Xi'an Jiaotong University, Xi'an 710049, China
  (\email{sblin1983@gmail.com}).}
  }

\usepackage{amsopn}
\DeclareMathOperator{\diag}{diag}

\newcommand{\sfgrad}[1][]{ 
	\nabla_{*}
}

\newcommand{\sfcurl}[1][]{ 
	\mathbf{L}
}

\newcommand{\Jw}[1][\alpha,\beta]{ 
w_{#1}
}

\newcommand{\imat}[1][d]{ 
    I
}

\newcommand{\Lpw}[2][\Jw]{ 
\mathbb{L}_{#2}(#1)
}

\newcommand{\InnerLGb}[2][{\Jw[r-\frac{1}{2},r-\frac{1}{2}]}]{ 
\left(#2\right)_{\Lpw[{#1}]{2}}
}

\newcommand{\Diff}[2][t]{ 
\ifthenelse{\equal{#2}{1}}{\frac{\mathrm{d}}{\mathrm{d}#1}}{
\left(\frac{\mathrm{d}}{\mathrm{d}#1}\right)^{#2}}
}

\ifpdf
\hypersetup{
  pdftitle={Integral Operator Approaches on Spheres},
  pdfauthor={S.-B. Lin}
}
\fi

\begin{document}

\maketitle
\begin{abstract} This paper focuses on scattered data fitting problems on spheres. We study the approximation performance of a class of weighted spectral filter algorithms (WSFA), including Tikhonov regularization, Landweber iteration, spectral cut-off, and iterated Tikhonov, in fitting noisy data with possibly unbounded random noise. For theoretical analysis, we develop an  integral operator approach  that can be regarded as an extension of the widely used sampling inequality approach and norming set method in the community of scattered data fitting. After providing an equivalence between the operator differences and quadrature rules, we succeed in deriving tight bounds for operator differences,  explicit operator representations for WSFA  and consequently optimal    error estimates. Our derived  error estimates do not suffer from the saturation phenomenon for Tikhonov regularization, native-space-barrier for existing error analysis and adapts to different embedding spaces. Based on the operator representations, we 
develop a Lepskii-type principle to determine the filter parameter of WSFA and a divide-and-conquer scheme to to reduce the computational burden and provide optimal approximation rates for corresponding algorithms.

\end{abstract}




\pagestyle{myheadings}
\thispagestyle{plain}

\section{Introduction}

Scattered data fitting on spheres, also known as spherical fitting or spherical regression, abounds in numerous applications such as 3D point cloud registration \cite{tam2012registration}, image rendering \cite{tsai2006all}, medical  imaging  \cite{sbibih2007new}, geophysics \cite{king2012lower}, planetary science \cite{wieczorek1998potential},   and signal recovery \cite{mcewen2011novel}. For example,  point clouds obtained from different sensors or scans may contain noise or outliers due to measurement errors.  Fitting scattered data  in the point cloud is critical to align and register multiple point clouds and enable accurate 3D reconstruction, object recognition, and augmented reality applications. Developing  powerful  strategies for spherical fitting 
not only helps practitioners to efficiently process data  but also deepens their  understanding on data  and consequently provides feasible guidance to  develop efficient sampling, storage, and processing techniques.

Given the set of scattered data $\Lambda=\{x_i\}_{i=1}^{|\Lambda|}\subset\mathbb S^d$ with $\mathbb S^d$ the  unit sphere embedded into the $d+1$-Euclidean space $\mathbb R^{d+1}$, 
we are concerned with noisy data fitting model  
\begin{equation}\label{Model1:fixed}
        y_{i}=f^*(x_{i})+ \varepsilon_{i},  \qquad\forall\
        i=1,\dots,|\Lambda|
\end{equation}
for sampling noise $\{\varepsilon_{i}\}_{i=1}^{|\Lambda|}\subset\mathbb R^{|\Lambda|}$  and target function $f^*:\mathbb S^d\rightarrow\mathbb R$. Due the existence of outliers, it is reasonable to assume
the noise to be  mean-zero random that can be unbounded. 
Writing $D=\{(x_i,y_i)\}_{i=1}^{|\Lambda|}$, our purpose is to derive an estimate $f_D$ based on $D$ such that $f_D$ is a good approximation of $f^*$. 

Kernel (minimal norm) interpolation (KI) \cite{hubbert2015spherical} is a classical and long-standing approach for spherical fitting.  For a positive definite function $\phi$,    KI is mathematically defined by
\begin{equation}\label{minimal-norm-interpolation}
   f_{D}:=  \arg\min\limits_{f\in\mathcal N_\phi}\|f\|_\phi,\qquad
    \mbox{s.t.} \quad 
     f(x_i)=y_{i},\quad i=1,\dots,|\Lambda|, 
\end{equation}
where $\mathcal N_\phi$ is the native space (also called as the reproducing kernel Hilbert space in the machine learning community \cite{cucker2007learning}) associated with $\phi$. The approximation performance of KI  for noise-free data, i.e., $\varepsilon_i=0$ in \eqref{Model1:fixed},   has been extensively studied in the literature \cite{narcowich2002scattered,levesley2005approximation,narcowich2007direct}, in which optimal approximation rates    were established under different metrics  even when $f^*$ is not in $\mathcal N_\phi$. These exciting results demonstrate the excellent  power of KI in fitting noise-free data. The problem is, however, that the powerful fitting capability inevitably enables KI to absorb the noise in the derived estimator when  data are contaminated, leading to the well known overfitting phenomenon in the sense of   deriving estimator fits the given data well but fails to accurately predict  new queries. Though the stability of KI,  measured by   condition numbers of   kernel matrices \cite{narcowich1998stability}, Lebesgue constants \cite{hangelbroek2010kernel} or $L^\infty$ norms of the $L^2$ projection operator \cite{hangelbroek2011kernel}, was well studied,  showing  that KI does not amplify the negative effects of  noise in the interpolation process, there lacks explicit approximation error estimates and controllable upper bounds for $\|f\|_\phi$ in tackling noisy data.
Furthermore, the
 inconsistency of KI  \cite{buchholz2022kernel,lin2023dis} shows that   the approximation error of  KI  is even not always decreasing with respect to the size of data.
 
This phenomenon was  noticed by \cite{hesse2017radial,Feng2021radial}, where (weighted) Tikhonov regularization   was employed to guarantee small (native space) norm of the derived estimator. 
The philosophy behind (weighted) Tikhonov regularization is to introduce a  regularization term to contral $\|f\|_\phi$. Mathematically, it is defined by
\begin{equation}\label{KRR}
    f^{Tik}_{D,\lambda,W}:=\arg\min_{f\in\mathcal N_\phi}\sum_{i=1}^{|\Lambda|}w_i(f(x_i)-y_i)^2+\lambda\|f\|_K^2 
\end{equation}
for $w_i>0$, $i=1,\dots,|\Lambda|$. Approximation rates of (weighted) Tikhonov regularization have been established in \cite{hesse2017radial,Feng2021radial}, exhibiting  excellent theoretical behaviors of (weighted) Tikhonov regularization. However,  results in \cite{hesse2017radial} are only available to extremely small noise but those in \cite{Feng2021radial} are built upon bounded mean-zero noise. Furthermore  approximation results in \cite{hesse2017radial,Feng2021radial} suffer  from the well known saturation phenomenon 
  \cite{gerfo2008spectral} in the sense that the derived approximation rates saturated for a certain smoothness of target functions 
  and the native-space-barrier (NSB) problem in terms that the approximation rates are derived under the assumption $f^*\in\mathcal N_\phi$, which is much worse than those for KI \cite{levesley2005approximation,narcowich2007direct}. The main reasons are two folds. On   one hand, (weighted) Tikhonov regularization is difficult to encode the smoothness of target functions up to a certain level due to its special spectral property  \cite{gerfo2008spectral}, resulting in its limited performance in  approximating over-smooth functions. On the other hand, there  lacks  suitable analysis tools for (weighted) Tikhonov regularization to tackle noisy data,  though several powerful methods such as the norming set approach \cite{jetter1999error},  Zero Lemma approach \cite{hangelbroek2012polyharmonic},   interpolation with best approximation approach \cite{narcowich2002scattered} and sampling inequality approach \cite{hesse2017radial}
  have been developed for KI to handle noise-free data. 

Motivated by \cite{gerfo2008spectral},   weighted spectral filter algorithms (WSFA)  including iterated Tikhonov, spectral cut-off and Landweber iteration have been proposed to  conquer the saturation phenomenon of Tikhonov regularization in our recent work \cite{liu2024weighted}. Moreover, the classical integral operator approaches \cite{smale2007learning,lin2017distributed} in machine learning were borrowed to analyze the  approximation performance of WSFA in \cite{liu2024weighted}. However, such an integral operator approach,  is  only feasible under  the well-specified setting: $f^*\in\mathcal N_\phi$.
 In this paper, we modify the integral operator approach in \cite{liu2024weighted} by building an equivalence between operator differences and spherical quadrature rules for products of functions to derive 
 approximation errors without saturation, adapting to different metrics, and circumventing NSB.
The main novelties in analysis are   quadrature rules for products of functions, and   consequently  tight bounds for numerous operator differences and a delicate operator representation to enable the analysis to accommodate the mis-specified case $f^*\notin\mathcal N_\phi$.

Our contributions are four aspects. At first,  we propose a novel integral operator approach including  delicate  operator representations for WSFA and exclusive operator differences to   circumvent NSB. In particular, we show that the proposed integral operator approach can be regarded as an essential extension of the sampling inequality method \cite{hesse2017radial}. Then, based on the derived operator representations, we  derive optimal approximation rates for  WSFA without saturation and NSB under different metrics. Thirdly, based on the developed operator representations and operator differences, we propose a Lepskii-type principle \cite{blanchard2019lepskii} to select the filter parameter and theoretically verify its optimality in terms of deriving optimal approximation error of corresponding WSFA. 
Finally,   we propose a   distributed approximation scheme based on divide-and-conquer to reduce the computational burden of WSFA and provide its optimal approximation rate. Our analysis  shows that such a divide-and-conquer  scheme does not degrade the approximation performance of WSFA under mild conditions. 
All these show the power of the integral operator approach in tackling noisy data on spheres.

The rest of paper is organized as follows. In the next section, we introduce spherical basis functions,   integral operators and some existing results for KI. In Section \ref{Sec:integral-operator}, we develop a novel analysis tool for noisy data fitting based on operator differences and quadrature rules. In Section \ref{Sec:integral-operator}, we formulate the equivalence between operator differences and quadrature rules for products of functions and study the relation between operator differences and  sampling mechanisms.
  In Section  \ref{sec.spectral}, we derive approximation error of WSFA and compare it with existing results. In Section \ref{Sec:parameter}, we develop a novel Lepskii-type principle to select the filter parameter of WSFA. 
  In Section \ref{sec:Dis}, we propose the distributed approximation scheme and derive its optimal approximation rates. In the last section, we present 
proofs of our  theoretical results.

\section{Spherical Basis Function   and Kernel Interpolation}\label{sec.SBF}
In this section, we introduce several basic concepts and properties of spherical harmonics, spherical basis functions, integral operators  and some existing results for KI.

\subsection{Spherical basis function and integral operators}
Let $L^2(\mathbb S^d)$ be the Hilbert space endowed with inner product
\[
\langle f, g\rangle_{L^2} = \int_{\mathbb S^d} f(x) g(x) d \omega(x), \quad f,g \in L^2(\mathbb S^d),
\]
where $d\omega$ denotes the scaled Lebesgue measure on $\mathbb S^d$, i.e., $\int_{\mathbb S^d}d\omega=1$. For $k\in\mathbb N$, denote by  $\mathbb{H}^{d}_k$  the subspace of $L^2(\mathbb S^d)$ consisting of all the   spherical harmonics of degree $k$, whose dimension \cite{muller2006spherical} is
\begin{equation}\label{dimension}
 Z(d,k)=\left\{\begin{array}{ll}
  \frac{2k+d-1}{k+d-1}{{k+d-1}\choose{k}}\stackrel{d}{\sim} k^{d-1}, & \mbox{if}\ k\geq 1, \\
1, & \mbox{if}\ k=0.
\end{array}
\right. 
\end{equation}
Throughout the paper, $a\stackrel{\nu}{\sim} b$ for $a,b,\nu>0$ denotes that there is a constant $C_\nu$ depending only on $\nu$ such that $C_\nu^{-1}b\leq a\leq C_\nu b$.
For a nonnegative integer $s$, denote by  $\mathcal P_s^{d}$ the set of  all   algebraic polynomials of degree at most $s$ defined on $\mathbb S^d$. Then we have 
$
   \mathcal P_s^{d}=\bigoplus_{k=0}^s\mathbb{H}^{d}_k$ and consequently 
$$
   \mbox{dim}\mathcal P_s^d = 
   \sum_{k=0}^sZ(d,k)=Z(d+1,s) \stackrel{d}{\sim}  s^d.
$$

A univariate  function $\phi$ supported on $[-1,1]$ is called a spherical basis function (SBF) \cite{narcowich2007direct}, if
its Gegenbauer series expansion
$
           \phi(u)=\sum_{k=0}^\infty
            \hat{\phi}_k\frac{Z(d,k)}{\Omega_d} P_k^{d+1}(u)
$ has all non-negative Gegenbauer-Fourier coefficients, i.e., 
$\hat{\phi}_k\geq 0$ for all $k=0,1,\ldots, $ where
  $P_k^{d+1}$ is the Gegenbauer polynomial of order $\frac{d-1}{2}$ and degree $k$ normalized so that $P_k^{d+1}(1)=1$ and $\Omega_d=\frac{2\pi^{\frac{d+1}{2}}}{\Gamma(\frac{d+1}{2})}$ denotes the volume of $\mathbb S^d$. If in addition $\hat{\phi}_k>0$ for all $k=0,1,\ldots,$ and  $\sum_{k=0}^\infty\hat{\phi}_k {Z(d,k)} <\infty$,   $\phi$ is  said to be a (strictly) positive definite function. 
Each SBF $\phi$ corresponds to a native space \cite{narcowich2002scattered}, defined by 
$$
               \mathcal N_\phi:=\left\{f(x)=\sum_{k=0}^\infty\sum_{\ell=1}^{Z(d,k)}\hat{f}_{k,\ell}Y_{k,\ell}(x): \|f\|_\phi^2:=
               \sum_{k=0}^\infty
         \hat{\phi}_k^{-1}\sum_{\ell=1}^{Z(d,k)}\hat{f}_{k,\ell}^2<\infty\right\}, 
$$  
endowed with the inner product $\langle f,g\rangle_{\phi}=\sum_{k=0}^\infty\hat{\phi}_k^{-1}\sum_{\ell=1}^{Z(d,k)}\hat{f}_{k,\ell}\hat{g}_{k,\ell}$, where $
                 \hat{f}_{k,\ell}:=\int_{\mathbb
                 S^d}f(x)Y_{k,\ell}(x)d\omega(x) 
$ denotes the Fourier coefficient for $f\in L^2(\mathbb S^d)$,
and  $\{Y_{k,\ell}\}_{\ell=1}^{Z(d,k)}$ is an orthonormal basis of $\mathbb H_k^d$ under the metric of $L^2(\mathbb S^d)$. It the follows from the definition of $\langle f,g\rangle_{\phi}$ that $\{\hat{\phi}_k^{-1/2}Y_{k,\ell}\}_{k=0,\ell=1}^{\infty,Z(d,k)}$ builds an arbitrary orthonormal basis for $\mathcal N_\phi$.
The well known addition formula \cite{muller2006spherical} describes the relation between    $P_k^{d+1}$ and $Y_{k,j}$ in terms that 
\begin{equation}\label{addition-formula}
             \sum_{\ell=1}^{Z(d,k)}Y_{k,\ell}(x)Y_{k,\ell}(x')=\frac{Z(d,k)}{\Omega_{d}}P_k^{d+1}(x\cdot x'), \qquad x, x' \in \mathbb S^d.
\end{equation}
Therefore, we get
\begin{equation}\label{kernel-basis-relation}
    \phi(x\cdot x')=\sum_{k=0}^\infty
            \hat{\phi}_k\sum_{\ell=1}^{Z(d,k)}Y_{k,\ell}(x)Y_{k,\ell}(x') 
\end{equation}
and 
\begin{equation}\label{Reproducing-property}
    \langle f,\phi_{x}\rangle_\phi=f(x),\qquad \forall f\in\mathcal N_\phi, 
\end{equation}
where $\phi_x(x')=\phi(x\cdot x')$.

Let 
 $\psi(\cdot)=\sum_{k=0}^\infty \hat{\psi}_k\frac{Z(d,k)}{\Omega_d} P_k^{d+1}(\cdot)$ and $\varphi(\cdot)=\sum_{k=0}^\infty
            \hat{\varphi}_k\frac{Z(d,k)}{\Omega_d} P_k^{d+1}(\cdot)$  
be another two SBFs satisfying 
\begin{equation}\label{kernel-relation}
     \hat{\psi}_k=\hat{\phi}_k^{\beta}, \hat{\varphi}_k=\hat{\phi}_k^{\alpha}, \qquad  \alpha\geq \beta,    0\leq \beta\leq 1.
\end{equation}
Write $J_{\phi,\psi}:\mathcal N_\phi\rightarrow \mathcal N_\psi$ as the canonical inclusion. The   Funk-Hecke formula \cite{muller2006spherical}
\begin{equation}\label{funkhecke}
                   \int_{\mathbb{S}^{d}}\phi(x\cdot x')Y_{k,j}(x')d\omega(x')=\hat{\phi}_kY_{k,j}(x),\quad\forall\ j=1,\dots,Z(d,k),k=0,1,\dots
\end{equation}
shows that for any $g\in\mathcal N_\psi$, $k=0,\dots,\infty$ and $\ell=1,\dots,Z(d,k)$, there holds
\begin{eqnarray}\label{relation11111}
      \int_{\mathbb S^d}\langle g,\phi_x\rangle_\psi Y_{k,\ell}(x)d\omega(x)
      &=&\sum_{k'=0}^\infty\sum_{\ell'=1}^{Z(d,k')} \int_{\mathbb S^d}\hat{\psi}_{k'}^{-1} \hat{g}_{k',\ell'}\int_{\mathbb S^d}\phi(x'\cdot x)Y_{k',j'}d\omega(x') Y_{k,\ell}(x)d\omega(x)\nonumber\\
      &=&
      \sum_{k'=0}^\infty\sum_{\ell'=1}^{Z(d,k')} \int_{\mathbb S^d}\hat{\psi}_{k'}^{-1} \hat{g}_{k',\ell'}\hat{\phi}_{k'}Y_{k',j'}(x) Y_{k,\ell}(x)d\omega(x)=\frac{\hat{\phi}_k}{\hat{\psi}_k} \hat{g}_{k,\ell}.
\end{eqnarray}
 Therefore, 
for any $f\in\mathcal N_\phi$, $g\in\mathcal N_\psi$,
 we have
\begin{eqnarray*} 
      \langle  J_{\phi,\psi}f,g\rangle_\psi
   &=&\sum_{k=0}^\infty \hat{\psi}_k^{-1}
\sum_{\ell=1}^{Z(d,k)}\hat{f}_{k,\ell}\hat{g}_{k,\ell}
 =
 \sum_{k=0}^\infty \hat{\phi}_k^{-1}\sum_{\ell=1}^{Z(d,k)}\hat{f}_{k,\ell}\frac{\hat{\phi}_k}{\hat{\psi}_k} \hat{g}_{k,\ell}\nonumber\\
 &=&
 \sum_{k=0}^\infty \hat{\phi}_k^{-1}\sum_{\ell=1}^{Z(d,k)}\hat{f}_{k,\ell}\widehat{(\langle g,\phi_x\rangle_\psi)}_{k,\ell}
 =\langle f,\langle g,\phi_x\rangle_\psi\rangle_\phi,
\end{eqnarray*}
showing that 
\begin{equation}\label{population-conuga}
    J_{\phi,\psi}^Tf(x)=\langle f,\phi_x\rangle_\psi,\qquad f\in\mathcal N_\psi.
\end{equation}
 Define further 
$\mathcal L_{\phi,\psi}=J_{\phi,\psi} J_{\phi,\psi}^T$ and $ L_{\phi,\psi}=J_{\phi,\psi}^TJ_{\phi,\psi}$. Then, $\mathcal L_{\phi,\psi}$ is an integral operator from $\mathcal N_\psi \rightarrow \mathcal N_\psi $ and $  L_{\phi,\psi}$ is an integral operator from $\mathcal N_\phi\rightarrow\mathcal N_\phi$.  In particular, $\mathcal L_{\phi,\phi}$ and $ L_{\phi,\phi}$ are the identity operator. 
If $\mathcal N_\psi=L^2(\mathbb S^d)$, we write $ J_{\phi,\psi}=:J_\phi$, $J_{\phi,\psi}^T=:J^T_\phi$,  $ \mathcal L_{\phi}=:\mathcal L_{\phi,\psi}$ and $L_{\phi}=:L_{\phi,\varphi}$ for the sake of brevity. In particular, $L_\phi$ denotes the integral operator 
$$
    L_\phi f(x)=\int_{\mathbb S^d}f(x')\phi(x\cdot x')d\omega(x').
$$


The following lemma  quantifies the relation between the smoothness of a function and the operator  $\mathcal L_{\phi,\psi}$.

\begin{lemma}\label{Lemma:source-condition}
If $\phi,\varphi,\psi$ satisfy \eqref{kernel-relation} with $\alpha\geq\beta$ and $0\leq\beta\leq 1$, then  
\begin{equation}\label{operator-relation-112}
    \mathcal L_{\phi,\psi}g=\mathcal L_{\phi}^{1-\beta}g,\qquad
    L_{\phi,\psi}f=L_\phi^{1-\beta}f,\qquad\forall g\in\mathcal N_\psi,f\in\mathcal N_\phi,
\end{equation}
and
\begin{equation}\label{norm-relation-123}
  \|f\|_\psi=\|L_\phi^\frac{1-\beta}2f\|_\phi=\|L_{\phi,\psi}^{1/2}f\|_\phi,\qquad\forall f\in\mathcal N_\phi. 
\end{equation}
Furthermore, for any $f\in\mathcal N_\varphi$,  there exists  $h\in L^2(\mathbb S^d)$ and $h'\in\mathcal N_\psi$ such that 
\begin{equation}\label{source-condition}
    f=\mathcal L_\phi^{\frac{\alpha-\beta}{2}}h\mathop{=}\limits^{\beta<1}  \mathcal L_{\phi,\psi}^{\frac{\alpha-\beta}{2(1-\beta)}} h,\qquad  \mbox{and} \quad 
    \|  f\|_\varphi=\|h\|_{L^2}=\| h'\|_\psi,
\end{equation}
where  $\mathop{=}\limits^{\beta<1}$ denotes the equality holds only for $\beta<1$,  and
$\eta(\mathcal L_{\phi})$ for any  $\eta:\mathbb R_+\rightarrow \mathbb R_+$  is defined by spectral calculus. 
\end{lemma}


For $w_i>0, i=1,\dots,|\Lambda|$,
define 
\begin{equation}\label{Sampling-operator}
   S_{\Lambda,W}f=(\sqrt{w_{1}}f(x_1),\dots,\sqrt{w_{|\Lambda|}}f(x_{|\Lambda|}))^T 
\end{equation}
 as the weighted sampling operator from $\mathcal N_\phi$ to $\mathbb R^{|\Lambda|}$.
Since for any $f\in\mathcal N_\phi$ and ${\bf c}\in\mathbb R^{|\Lambda|}$  there holds 
$$
  \langle S_{\Lambda,W}f,{\bf c}\rangle_{\mathbb R^{|\Lambda|}}
  =
  \sum_{i=1}^{|\Lambda|}c_i\sqrt{w_{i}}f(x_i)
  =\langle f,\sum_{i=1}^{|\Lambda|}c_i\sqrt{w_{i}}\phi_{x_i}\rangle_\phi=
  \langle f, S^T_{\Lambda,W}{\bf c}\rangle_\phi,
$$
where
\begin{equation}\label{sampling-conjugate}
    S^T_{\Lambda,W}{\bf c}:= \sum_{i=1}^{|\Lambda|}\sqrt{w}_{i}c_i\phi_{x_i},\qquad {\bf c}=(c_1,\dots,c_{|\Lambda|})^T.
\end{equation}
Hence, $S^T_{\Lambda,W}:\mathbb R^{|\Lambda|}\rightarrow\mathcal N_\phi$ is the adjoint of $S_{\Lambda,W}$. 
Define  $ L_{\phi,\Lambda,W}:\mathcal N_\phi\rightarrow \mathcal N_\phi$ by
\begin{equation}\label{weight-empirical-operaotr}
L_{\phi,\Lambda,W}f:=\sum_{i=1}^{|\Lambda|}w_{i}f(x_i)\phi_{x_i}=S_{\Lambda,W}^TS_{\Lambda,W}f 
\end{equation} 
and
 $ \mathcal L_{\phi,\psi,\Lambda,W}:\mathcal N_\psi\rightarrow\mathcal N_\phi$ by
\begin{equation}\label{weight-empirical-operaotr-1}
     \mathcal L_{\phi,\psi,\Lambda,W}f:=\sum_{i=1}^{|\Lambda|}w_{i}f(x_i)\phi_{x_i}.
\end{equation}
If $\beta=0$, we write $\mathcal L_{\phi,\Lambda,W}=:\mathcal L_{\phi,\psi,\Lambda,W}$ for the sake of brevity.
Then $L_{\phi,\Lambda,W}$, $\mathcal L_{ \phi,\psi,\Lambda,W}$ and $\mathcal L_{ \phi,\Lambda,W}$  can be regarded as  empirical versions of  $L_\phi$, $J_{\phi,\psi}^T$ and $J_{\phi}^T$, respectively. Recalling the weighted kernel matrix 
\begin{equation}\label{def.psi}
\Psi_{\Lambda,W}:=W^{1/2}\Phi_{\Lambda}W^{1/2}=S_{\Lambda,W} S^T_{\Lambda,W}
\end{equation} 
for $\Phi_{\Lambda}=(\phi(x_i\cdot x_j))_{i,j=1}^{|\Lambda|}$ and $W:=\mbox{diag}(w_{1},\dots,w_{|\Lambda|})$, we obtain from \cite[Proposition 9]{rosasco2010learning} that the spectra of $\Psi_{\Lambda,W}$ and $L_{\phi,\Lambda,W}$ are the same up to the zero, that is 
\begin{equation}\label{spectral-relation}
   \sigma_\ell(\Psi_{\Lambda,W})
=\sigma_\ell(L_{\phi,\Lambda,W}),\qquad \ell=1,\dots,|\Lambda|,
\end{equation}
where  $\sigma_\ell(A)$ denotes the $\ell$-th largest eigenvalue of the matrix or operator $A$.  It is easy to derive
\begin{equation}\label{operator-norm-bound}
    \|L_\phi\|_{\phi\rightarrow\phi},\|L_{\phi,\Lambda,W_s}\|_{\phi\rightarrow\phi}, 
    \|\mathcal L_{\phi,\psi,\Lambda,W}\|_{\psi\rightarrow\phi},\|J_{\phi,\psi}^T\|_{\psi\rightarrow\phi}\leq \kappa,
\end{equation}
where $\kappa:= \sqrt{\phi(1)} $ and $\|A\|_{\psi\rightarrow\phi}$  denotes  the spectral norm of the operator $A:\mathcal N_\psi\rightarrow\mathcal N_\phi$.
Moreover, all the mentioned operators in Table \ref{Table:notation-operators}  are trace-class,  and consequently Hilbert-Schmidt and compact. 

 \begin{table}[H]\label{Table:notation-operators}
    \begin{center}
	\begin{tabular}{|c|c|c|c|c|c|}
		\hline
		 \multicolumn{3}{|c|} {trace-class operators}   &  \multicolumn{3}{|c|} {positive operators}   \\
		\hline Notation & Range  &Definition &Notation & Range &Definition\\
  \hline
$J_{\phi,\psi}$ & $\mathcal N_\phi\rightarrow \mathcal N_\psi$ &  inclusion  & $L_{\phi,\psi}$   & $\mathcal N_\phi\rightarrow\mathcal N_\phi$& $J_{\phi,\psi}^TJ_{\phi,\psi}$ \\
   \hline
 $J_{\phi,\psi}^T$ & $\mathcal N_\psi\rightarrow\mathcal N_\phi$& \eqref{population-conuga} &  $\mathcal L_{\phi,\psi}$ & $\mathcal N_\psi\rightarrow \mathcal N_\psi$& $J_{\phi,\psi} J_{\phi,\psi}^T$\\
 \hline
$J_{\phi}$ & $\mathcal N_\phi\rightarrow L^2$ &  $J_{\phi,\psi}$, $\beta=0$  &  $L_{\phi}$   & $\mathcal N_\phi\rightarrow\mathcal N_\phi$ & $L_{\phi,\psi},\beta=0$\\
\hline
$J_{\phi}^T$ & $ L^2\rightarrow\mathcal N_\phi$&  $J_{\phi,\psi}^T,\beta=0$& $\mathcal L_{\phi}$   & $L^2\rightarrow L^2$& $\mathcal L_{\phi,\psi}, \beta=0$\\
\hline
 $S_{\Lambda,W}$ & $\mathcal N_\phi\rightarrow \mathbb R^{|\Lambda|}$& \eqref{Sampling-operator} &$L_{\phi,\Lambda,W}$ & $\mathcal N_\phi\rightarrow\mathcal N_\phi$ & $S_{\Lambda,W}^TS_{\Lambda,W}$ \\
 \hline
 $S_{\Lambda,W}^T$ & $\mathbb R^{|\Lambda|}\rightarrow \mathcal N_\phi$ &\eqref{sampling-conjugate}  &  &  &   \\
  \hline
 $\mathcal L_{\phi,\psi,\Lambda,W}$& $\mathcal N_\psi\rightarrow\mathcal N_\phi$ & \eqref{weight-empirical-operaotr-1} &  &  &   \\
\hline
$\mathcal L_{\phi,\Lambda,W}$  &  $L^2\rightarrow\mathcal N_\phi$ &   $\mathcal L_{\phi,\psi,\Lambda,W}, \beta=0$ &  &  &   \\
\hline 
    \end{tabular}
    \end{center}
      \caption{Notations of operators} 
    \end{table}

\subsection{SBF interpolation}

Based on  \eqref{minimal-norm-interpolation}, it is easy to derive
\begin{equation}\label{KI}
    f_{D}=\sum_{i=1}^{|\Lambda|}a_i\phi_{x_i}, \qquad \mbox{with}\
   (a_1,\dots,a_{|\Lambda|})^T= :{\bf a}_{D}=\Phi_D^{-1}{\bf y}_D 
\end{equation} 
with  ${\bf y}_D=(y_1,\ldots, y_{|\Lambda|})^T$. 
Therefore, for any diagonal matrix  $W:=\mbox{diag}(w_{1},\dots,w_{|\Lambda|})$ with $w_i>0$, direct computation together with \eqref{def.psi} yields
\begin{equation}\label{Operator-representation:KI}
    f_D=S_{\Lambda,W}^T(\Psi_{D,W})^{-1}{\bf y}_{\Lambda,W}=S_{\Lambda,W}^T(S_{\Lambda,W},S_{\Lambda,W}^T)^{-1}{\bf y}_{\Lambda,W}
\end{equation}
for
   ${\bf y}_{\Lambda,W}:=( \sqrt{w_{1}}y_1,\dots,\sqrt{w_{|\Lambda|}}y_{|\Lambda|})^T.
$ 
Define further 
\begin{equation}\label{clean-KI}
     f^\diamond_D:=\mathcal I_D f^*:=\sum_{i=1}^{|\Lambda|}a_i^*\phi_{x_i},\qquad \mbox{with}\  (a_1^*,\dots,a^*_{|\Lambda|})^T= :{\bf a}_{D}^*=\Phi_D^{-1}{\bf f}^*
\end{equation}
for ${\bf f}^*=(f^*(x_1),\dots,f^*(x_{|\Lambda|}))^T$ as the noisy free version of $f_D$. If $f^*\in\mathcal N_\phi$, then
\begin{equation}\label{clean-KI-operator}
    \mathcal I_D f^*=S_{\Lambda,W}^T(S_{\Lambda,W},S_{\Lambda,W}^T)^{-1}S_{\Lambda,W} {\bf f}^*
\end{equation}
 
The approximation performance of KI can be measured by distribution of scattered data: mesh norm, separation radius and mesh ratio. 
 The mesh norm   of $\Lambda$, defined  by
$
                 h_{\Lambda}:=\max_{ x\in\mathbb S^d}\min_{ x_{i}\in \Lambda}\mbox{dist}( x,x_{i}), 
$
quantifies the smallest radius of spherical caps with centers in $\Lambda$ that can cover $\mathbb S^d$, where $\mbox{dist}(x,x')$ denotes the geodesic distance between $x,x'\in\mathbb S^d$. The  separation radius, given as $q_{\Lambda}:=\frac12\min_{i\neq i'}\mbox{dist}( x_{i}, x_{i'})$, describes the minimal distance between elements in $\Lambda$. The mesh ratio, written as $\rho_{\Lambda}:=\frac{h_{\Lambda}}{q_{\Lambda}}\geq 1$, measures how uniformly the points of $\Lambda$ are distributed on $\mathbb S^d$. 
For a given $\tau\geq1$,  
  $\Lambda$ is  said to be $\tau$-quasi-uniform  if  $\rho_{\Lambda}\leq \tau$.
The following approximation rates for KI derived in \cite{levesley2005approximation,narcowich2007direct} shows the performance of $f^\diamond_D$ in fitting noise-free data.

\begin{lemma}\label{Lemma:Narcowich-interpolation}
Suppose that $\phi$ is an SBF with $\hat{\phi}_k\stackrel{d}\sim k^{-2\gamma}$ satisfying $2\gamma>d$ and \eqref{kernel-relation} holds with $0<\beta\leq \min\{\alpha,1\}$ and $\alpha\gamma>d/2$.
If $f^*\in \mathcal N_\varphi$, then  
\begin{equation}\label{Error-est-noiseless}
     \|f^*-f_D^\diamond\|_{\psi}\leq C \rho_\Lambda^{\max\{1-\alpha,0\}}
     h_\Lambda^{(\alpha-\beta)\gamma},
\end{equation}
 where $C$ is a constant depending only on $d, \alpha, \beta,$ and $\gamma$.   
\end{lemma}

The error analysis in Lemma  \ref{Lemma:Narcowich-interpolation} is nice in four aspects. At first, the derived approximation rates are optimal in the sense that, for $\tau$-quasi-uniform  $\Lambda$,  the approximation rates of order $\mathcal O(|\Lambda|^{-\frac{(\alpha-\beta)\gamma}{d}})$    cannot be improved further due to the inverse theorem established in \cite{narcowich2007direct}. Then, Lemma \ref{Lemma:Narcowich-interpolation} succeeds in conquer NSB such that    $f^*$ can be out of $\mathcal N_\phi$, i.e. $\alpha<1$. Thirdly, the error analysis is carried under numerous metrics, i.e. $\|\cdot\|_\psi$ for any $\psi$ satisfying \eqref{kernel-relation} with $0\leq \beta\leq \min\{1,\alpha\}$, showing its adaptivity to the embedding spaces.
At last, there is not any saturation phenomenon \cite{gerfo2008spectral} for $f_D$ in terms that the approximation rates derived in \eqref{Error-est-noiseless} cover  all range of $\alpha\geq1$. Besides the approximation error in Lemma \ref{Lemma:Narcowich-interpolation}, the stability of KI has also been extensively studied in terms of minimal eigenvalue of the kernel matrix \cite{narcowich1998stability}, Lebesgue constant \cite{hangelbroek2010kernel}, i.e. $\|\mathcal I_D\|_{L^\infty\rightarrow L^\infty}$, and $L^\infty$ norm of the $L^2$ projection \cite{hangelbroek2011kernel}, i.e. $\|\mathcal I_D\|_{L^2(\mathbb S^d)\rightarrow L^\infty(\mathbb S^d)}$. For example, it can be found in \cite{hangelbroek2010kernel,hangelbroek2012polyharmonic} that the Lebesgue constant of KI for some special kernels can be bounded by a constant, provided $\Lambda$ is $\tau$-quasi-uniform, showing that KI does not amplify the negative effect of the noise.

However,  facing with scattered data for large noise, KI is frequently not efficient  since the approximation error sometimes increases with respect to the size of data \cite{liu2024weighted,lin2023dis}. In fact, it was deduced in \cite{lin2023dis} the following inconsistency result for KI \eqref{minimal-norm-interpolation}.
\begin{lemma}\label{Lemma:inconsistence}
Let $\Lambda$ be $\tau$-quasi-uniform for some $\tau>1$, $\hat \phi_k\stackrel{d}\sim k^{-2\gamma}$  with $\gamma> d/2$ and $\{\varepsilon_i^*\}_{i=1}^{|\Lambda|}$ be a set of random variables whose supports are contained in $[-M', -\theta M'] \cup [\theta M', M'] $ for some $0< \theta <1$ and $M'>0$.
Then for all $f^* \in \mathcal N_\phi$ with $f^*|_\Lambda =0,$ there holds almost surely  
\begin{equation}
     \|f_{D}^\diamond-f^*\|_\phi\geq \tilde{C},
\end{equation}
where $f_D^\diamond$ is defined by \eqref{clean-KI}   and
$\tilde{C}>0$  depends only on $\theta, \tau,\gamma,M'$ and $d$.
\end{lemma}
 
Lemma \ref{Lemma:inconsistence} demonstrates that there always exists a ``bad'' $f^*\in\mathcal N_\phi$ such that $f_D^\diamond$ performs not  well, provided the  noise is relatively large.  
It should be highlighted that  Lemma \ref{Lemma:inconsistence} does not contradict with the existing stability analysis 
\cite{narcowich1998stability,hangelbroek2010kernel,hangelbroek2011kernel,hangelbroek2012polyharmonic}. 
To be detailed, the noise presented in Lemma \ref{Lemma:inconsistence} is large, implying that KI, even with small Lebesgue constant, cannot provide   good approximation of the target function, which is  known as the over-fitting phenomenon in the community of machine  theory \cite{cucker2007learning}.

\section{Operator Differences and Sampling Mechanisms}\label{Sec:integral-operator}
The integral operator approach, originally proposed by \cite{smale2004shannon,smale2005shannon} in machine learning for random samples, aims to quantify the performance of learning algorithms. Via operator representations   and operator differences based on concentration inequalities  in probability. As spherical data are frequently deterministically sampled, there lack corresponding concentration inequalities to derive tight  bounds for operator differences. In this section, we show that   quadrature rules and delicate sampling mechanisms are sufficient to mimic  concentration inequalities in probability for random samples.

\subsection{Operator differences}
Numerous elegant analysis tools, including the norming set approach \cite{jetter1999error,levesley2005approximation}, best approximation of polynomial interpolants method \cite{narcowich2002scattered,narcowich2007direct}, Zeros Lemma strategy \cite{hangelbroek2010kernel,hangelbroek2012polyharmonic} and  sampling inequalities \cite{hesse2017radial}, have  been developed for scattered data fitting on spheres. Besides  sampling inequalities, all these  tools are only available to  noise-free data. For example, Zeros Lemma presented in \cite[Corollary A.13]{hangelbroek2012polyharmonic} shows that if $\hat{\phi}_k\stackrel{d}\sim k^{-2\gamma}$  with $\gamma>d/2$ and \eqref{kernel-relation} holds for $\alpha\geq\beta$  and  $\alpha\gamma>d/2$, then for any $\Lambda$ with sufficiently small $h_\Lambda$ and any   $f\in\mathcal N_\varphi$ satisfying $f|_\Lambda=0$, there holds
\begin{equation}\label{Zeros-lemma}
    \|f\|_\psi\leq C_1 h_\Lambda^{(\alpha-\beta)\gamma}\|f\|_{\varphi},
\end{equation}
which implies Lemma \ref{Lemma:Narcowich-interpolation} directly. 

Sampling inequalities demonstrate that for smooth functions, their $L^2$ norms cannot be large if their  empirical $\ell^2$ norms are small. In particular, under the same setting as above, it can be found in \cite[Theorem 5.1]{hesse2017radial} that for any $f\in\mathcal N_\phi$, there holds
\begin{equation}\label{Sampling-inequality-1}
    \|f\|_{L^2}\leq C_2\left(h_\Lambda^\gamma\|f\|_\phi+
    h_\Lambda^{d/2}|\Lambda|^{1/2}\|f\|_\Lambda\right),
\end{equation}
where $\|f\|_\Lambda^2=\frac1{|\Lambda|}\sum_{i=1}^{|\Lambda|}|f(x_i)|^2$.
The established sampling inequality is available to noisy data by adding an additional empirical norm of $f$. However, the inequality \eqref{Sampling-inequality-1} only holds for $f\in\mathcal N_\phi$ under the  $L^2$ norm. 

The basic idea of the integral operator approach, as shown in \cite{smale2004shannon,smale2005shannon,smale2007learning,caponnetto2007optimal} for machine learning and \cite{Feng2021radial,lin2021subsampling,lin2023dis,liu2024weighted} for spherical data fitting, is to present tight bounds for three operator  (or functional) differences. The first one is 
\begin{equation}\label{Def.RD}
     \mathcal R_{\Lambda,\lambda,W,u,v} := \|(L_\phi+\lambda I)^{-u}(L_\phi-L_{\phi,\Lambda,W})(L_\phi+\lambda I)^{-v}\|_{\phi\rightarrow\phi},\qquad 0\leq u,v\leq 1
\end{equation}
that quantifies the difference between $L_\phi$ and its empirical version $L_{\phi,\Lambda,W}$. 
The second one is 
\begin{equation}\label{def:S}
    \mathcal S_{\Lambda,\lambda,W,u}:=\|( L_\phi+\lambda I)^{-u}
  (\mathcal  L_{\phi,\Lambda,W} -J_{\phi,\psi}^T)  \|_{0\rightarrow\phi}, \qquad 0\leq u\leq 1
\end{equation}
to measure  the similarity of $J_{\phi,\psi}^T$ and its empirical counterpart $\mathcal  L_{\phi,\Lambda,W}$, where $\|\cdot\|_{0\rightarrow\phi}$ denotes $\|\cdot\|_{L^2\rightarrow \mathcal N_\phi}$.
The last one, mainly introduced to measure the negative effect of noise, is given by 
\begin{equation}\label{def:P} 
     \mathcal P_{D,\lambda,W,\beta} := \|(  L_\phi+\lambda I)^{-1/2}(S^T_{\Lambda,W} {\bf y}_{\Lambda,W}-\mathcal L_{\phi,\psi, \Lambda,W} f^*)\|_\phi,
\end{equation}
 referring to the difference between $\mathcal L_{\phi,\psi ,\Lambda,W} f^*$ and its empirical version $S^T_{\Lambda,W} {\bf y}_{\Lambda,W}$. It follows from \eqref{sampling-conjugate} and \eqref{weight-empirical-operaotr-1} directly that $\varepsilon_i=0$ in \eqref{Model1:fixed} implies $\mathcal P_{D,\lambda,W,\beta}=0$. In the following proposition, we show that tight bounds of operator differences imply sampling inequalities.


\begin{proposition}\label{Prop:operator-sampling}
Let $W=\diag(w_1,\dots,w_{|\Lambda|})$ with $w_i>0$ and $\lambda>0$. If \eqref{kernel-relation} holds for some $0\leq \beta\leq 1$ and 
$\mathcal R_{\Lambda,\lambda,W,1/2,1/2}\leq \tilde{c}<1/2$ for some $\tilde{c}>0$, then  
\begin{eqnarray}
     \mathcal Q_{\Lambda,\lambda,W}&:=& \| (L_\phi+\lambda I)^{1/2} (L_{\phi,\Lambda,W}+\lambda I)^{-1/2}\|_{\phi\rightarrow\phi} \leq \sqrt{\frac{1}{1-\tilde{c}}},\label{def:Q}\\
     \mathcal Q^*_{\Lambda,\lambda,W}&:=& \|( L_\phi+\lambda I)^{-1/2} (L_{\phi,\Lambda,W}+\lambda I)^{1/2}\|_{\phi\rightarrow\phi}
     \leq\sqrt{\frac{1-\tilde{c}}{1-2\tilde{c}}},\label{def:Q*}\\
     \| f\|_{\psi}  &\leq&
    \frac{1}{(1-\tilde{c})^{(1-\beta)/2}}  \lambda^{-\beta/2} (\lambda^{1/2}\|f\|_\phi+\|f\|_{\Lambda,W}), \qquad f\in\mathcal N_\phi, \label{sampling-inequality-sobolev-1.1}\\
    \|f\|_{\Lambda,W}  &\leq&
     \sqrt{\frac{1-\tilde{c}}{1-2\tilde{c}}}(\|f\|_{L^2}+\lambda^{1/2}\|f\|_\phi),  \qquad f\in\mathcal N_\phi, \label{sampling-inequality-sobolev-2.1}
\end{eqnarray}
where $\|f\|_{\Lambda,W}^2:=\frac1{|\Lambda|}\sum_{i=1}^{|\Lambda|}w_i|f(x_i)|^2$ denotes the empirical weighted  $\ell^2$-norm of $f$.
\end{proposition}

The inequality \eqref{sampling-inequality-sobolev-1.1} improves the results in \cite{hesse2017radial} by presenting Sobolev-type sampling inequalities. Furthermore, \eqref{sampling-inequality-sobolev-2.1} yields an inverse sampling inequalities, demonstrating that the empirical weighted $\ell^2$ norm of a smooth function cannot be large if its $L^2$ norm  is small. Combining \eqref{sampling-inequality-sobolev-1.1} and \eqref{sampling-inequality-sobolev-2.1}, we get for any $f\in\mathcal N_\phi$
\begin{eqnarray}\label{Holder-type-norm}
   \| f\|_{\psi} 
   &\leq&
   \frac{1}{(1-\tilde{c})^{(1-\beta)/2}}\left(1+\sqrt{\frac{1-\tilde{c}}{1-2\tilde{c}}}\right)  \lambda^{-\beta/2} 
   \left(\lambda^{1/2}\|f\|_\phi+ \|f\|_{L^2} \right)\nonumber\\
   &\leq&
   \frac{1}{(1-\tilde{c})^{(1-\beta)/2}}\left(1+\sqrt{\frac{1-\tilde{c}}{1-2\tilde{c}}}\right)  \lambda^{-\beta/2} 
    \|(L_\phi+\lambda I)^{1/2}f\|_\phi, 
\end{eqnarray}
quantifying the relationship among different norms.

\subsection{Bounds of operator differences} 
This subsection devotes to bounding operator (or functional) differences defined by \eqref{Def.RD}, \eqref{def:S} and \eqref{def:P}. As $P_{D,\lambda,W,\beta}$ depends on the noise in model \eqref{Model1:fixed}, we introduce the following definition of sug-Gaussian noise. 
 \begin{definition}\label{Def:sub-Gaussian}
 A mean-zero random variable $\varepsilon$ is said to be  sub-Gaussian, if there is a $K\geq0$ (called as the width) such that 
\begin{equation}\label{moment-for-sub-Gaussian}
    E\{\exp(u \varepsilon)\}\leq \exp(Ku^2), \qquad \forall u\in\mathbb R.
\end{equation}
Furthermore, the sub-Gaussian norm of $\varepsilon$ is defined by 
\begin{equation}\label{sub-Gaussian-norm}
    \|\varepsilon\|_{SG}:=\inf\{t>0:E\{\exp(\varepsilon^2/t^2)\}\leq 2.
\end{equation}
\end{definition}

It was shown in \cite[Prop.2.5.2]{vershynin2018high} that  for an arbitrary sub-Gaussian random variable $\varepsilon$, its  sub-Gaussian norm only depends on its width and there are numerous equivalent definitions of the sub-Gaussian random variable. In particular, \eqref{moment-for-sub-Gaussian} implies
\begin{equation}\label{moment-for-squre-sub-gau-1}
  E\{\exp(\mu^2\varepsilon^2)\} \leq \exp(C_3\mu^2),\qquad \forall |\mu|\leq 1/C_3
\end{equation}
and
\begin{equation}\label{moment-for-squre-sub-gau-2}
    E\{\exp(\varepsilon^2/(C_4)^2)\}\leq 2,
\end{equation}
 where $C_3,C_4$ are constants depending only on $K$.
It can be found in \cite[Example 2.5.8]{vershynin2018high} that
if $\varepsilon$ is a mean-zero Gaussian random variable with variance $\sigma^2$, a random variable with symmetric Bernoulli distribution, or a random variable with $\|\varepsilon\|_\infty \leq M$, then it is also a sub-Gaussian variable with sub-Gaussian norms  $C_5\sigma$, $\frac{1}{\sqrt{\ln 2}}$ and $\frac{M}{\sqrt{\ln 2}}$, respectively, where $C_5$ is an absolute constant. With this, we present the first assumption of our study.
\begin{assumption}\label{Assumption:data}
Let $D=\{(x_i,y_i)\}_{i=1}^{|\Lambda|}$ be the set of data satisfying \eqref{Model1:fixed}   with $f^*\in\mathcal N_\varphi$ for $\varphi$ and $\phi$ satisfying \eqref{kernel-relation} and  $\{\varepsilon_i\}_{i=1}^{|\Lambda|}$  being a set of mean-zero sub-Gaussian  random variables for sub-Gaussian norm $M>0$.
\end{assumption}

Assumption \ref{Assumption:data} describes the regularity of  target functions  via $f^*\in\mathcal N_\varphi$ and the property of noise, i.e., mean-zero sub-Gaussian noise. As discussed above, there are quite mild in the community of spherical fitting \cite{narcowich2002scattered,narcowich1998stability,hesse2017radial,Feng2021radial,lin2017distributed,liu2024weighted}.
The following proposition presents a tight bound for $P_{D,\lambda,W}$.

\begin{proposition}\label{Proposition:value-difference-deterministic}
 Let $0<\delta<1$, $\hat{\phi}_k\stackrel{d}\sim k^{-2\gamma}$ with $\gamma>d/2$ and $\phi$, $\psi$, $\varphi$ satisfy \eqref{kernel-relation} for $0\leq \beta\leq 1$, $\alpha\geq \beta$. If Assumption \ref{Assumption:data} holds, 
 then   with confidence $1-\delta$,  
\begin{eqnarray}\label{bound-P-random}
     \mathcal P_{D,\lambda,W,\beta} 
     \leq     
     C_6\lambda^{-d/(4\gamma)} \left(\sum_{i=1}^{|\Lambda|}w_i^2+w_{\max}^2\right)^{1/2} \log\frac6\delta , 
\end{eqnarray} 
where 
 $w_{\max}=\max_{i=1,\dots,|\Lambda|}w_i$ and $ C_6$ is a  constant    depending only on $d$, $\gamma, M$.
\end{proposition}

To derive bounds of $\mathcal R_{\Lambda,\lambda,W,u,v}$ and $\mathcal S_{\Lambda,\lambda,W,u}$,  define the discrepancy   of quadrature  by 
\begin{eqnarray}
{WCE}^{\varphi,\phi}_{\Lambda,\lambda,W,u}&:=&\sup_{\|f\|_{\varphi}\leq 1}\sup_{\|g\|_\phi\leq 1}
  \left|\int_{\mathbb S^d} f(x')\eta_{\lambda,u}(  L_\phi)g(x')d\omega(x')\right. \nonumber\\
  &-&
\left.\sum_{i=1}^{|\Lambda|}w_{i}f(x_i)\eta_{\lambda,u}(L_\phi)g(x_i)\right|   \label{def.WCE1}\\
{WCE}^{\phi,\phi}_{\Lambda,\lambda,W,u,v}&:=&\sup_{\|f\|_{\phi}\leq 1}\sup_{\|g\|_\phi\leq 1}
  \left|\int_{\mathbb S^d} \eta_{\lambda,v}(    L_{\phi})f(x')\eta_{\lambda,u}(  L_\phi)g(x')d\omega(x')\right. \nonumber\\
  &-&
\left.\sum_{i=1}^{|\Lambda|}w_{i}\eta_{\lambda,v}(  L_{\phi})f(x_i)\eta_{\lambda,u}(L_\phi)g(x_i)\right|  \label{def.WCE2}
\end{eqnarray}
for  $\lambda>0$, $0\leq u,v\leq 1$ and $\varphi,\psi$ satisfying \eqref{kernel-relation} with $\alpha\geq\beta\geq0$,  
where  $\eta_{\lambda,u}(A)f=(A+\lambda I)^{-u}f$  for positive operator $A$. 
The following proposition presents the relation between operator differences and  the above two quantities.

\begin{proposition}\label{prop:equivalance-11}
For any    $ w_i>0$, $\lambda>0$ and $0\leq u,v\leq 1$, there holds
\begin{eqnarray}
    \mathcal R_{\Lambda,\lambda,W,u,v} &=&{WCE}^{\phi,\phi}_{\Lambda,\lambda,W,u,v},\label{equivalance-1}\\
    \mathcal S_{\Lambda,\lambda,W,u} &=&{WCE}^{\varphi,\phi}_{\Lambda,\lambda,W,u}.\label{equivalance-2}
\end{eqnarray}
\end{proposition}

\subsection{Sampling and weighting schemes}
Proposition \ref{prop:equivalance-11} and Proposition \ref{Proposition:value-difference-deterministic} present two requirements for the sampling and weighting mechanisms to guarantee small differences between integral operators (functional) and their empirical counterparts, i.e., tight bounds for quadrature rules of products of functions and small $w_{\max}$. In this subsection, we study the sampling and weighting schemes for this purpose, though their existence has bee well studied in the literature \cite{dai2016riesz,buhmann2021discretization,Feng2021radial}. To present a detailed construction, we need the following definition of quadrature rules.

\begin{definition}\label{Def:quadrature-rule}
For   $\Lambda=\{x_i\}_{i=1}^{|\Lambda|}\subseteq\mathbb S^d$, $s\in\mathbb N$ and $\{w_{i,s}\}_{i=1}^{|\Lambda|}\subset\mathbb R_+^{|\Lambda|}$, if 
\begin{equation}\label{eq:quadrature}
    \int_{\mathbb S^d}P(x)d\omega(x)=\sum_{x_{i}\in\Lambda} w_{i,s} P(  x_{i}), \qquad \forall P\in \mathcal P_s^d.
\end{equation}
then 
 $\mathcal Q_{\Lambda,s}:=\{(w_{i,s},  x_{i}): w_{i,s}> 0
\hbox{~and~}   x_{i}\in \Lambda\}$ is said to be a 
 spherical   positive
 quadrature rule  of degree $s \in\mathbb N$. If in addition the weights $\{w_{i,s}\}_{i=1}^{|\Lambda|}$ satisfy $0< w_{i,s}\leq c_*{|\Lambda|^{-1}}$, then   $\mathcal Q_{\Lambda,s}$ is said to be a  D-type quadrature rule of degree $s$ and   coefficient $c_*>0$. If $ w_{i,s}=\frac1{|\Lambda|}$ for any $i=1,\dots,|\Lambda|$, $\mathcal Q_{\Lambda,s}$ is said to a spherical $s$-design. A spherical $s$-design is said to be $c_0$-tight if $|\Lambda|=c_0s^d$, where $c_0$ is a positive constant  depending only on $d$  
\end{definition}

The constructions of quadrature weights for scattered data, D-type quadrature rules for quasi-uniform sets and spherical $s$-design can be found in \cite{mhaskar2001spherical,le2008localized}, \cite{brown2005approximation} and \cite{bannai2009survey,womersley2018efficient}, respectively. The existence of $c_0$-tight spherical $s$-design was verified in \cite{bondarenko2013optimal}. 
The following lemma, derived in \cite{mhaskar2001spherical,brown2005approximation} presents the existence of quadrature rules.

\begin{lemma}\label{Lemma:fixed cubature-poly}
If there exists a constant $\tilde{c}$ such that $h_\Lambda\leq \tilde{c}$ and $s\leq \tilde{c}'h_\Lambda^{-1}$ for some $\tilde{c}'>0$, then \eqref{eq:quadrature} holds with $\sum_{i=1}^{|\Lambda|}w_{i,s}^2\leq |\Lambda|^{-1}$. If in addition $\Lambda$ is $\tau$-quasi-uniform for some $\tau\geq 1$ and $s\leq c_\diamond |\Lambda|^{1/d}$, then \eqref{eq:quadrature} holds with $0<w_i\leq c_*|\Lambda|^{-1}$, where $c_*,c_\diamond$ are constants  depending only on $\tau$ and $d$.
\end{lemma}

Furthermore, the following lemma \cite{bondarenko2013optimal} shows the existence of $c_0$-tight $s$-spherical design.

\begin{lemma}\label{Lemma:Spherica-d-1}
 There exists a constant $c_0>0$, depending only on $d$, such that for every $|\Lambda|\in\mathbb N$ and $s\geq 1$, there exists a $c_0$-tight $s$-spherical  design.
\end{lemma}

Before presenting sampling and weighting mechanisms, we derive upper  bounds of ${WCE}^{\varphi,\phi}_{\Lambda,\lambda,W_s,u_1,u_2}$ and ${WCE}^{\phi,\phi}_{\Lambda,\lambda,W_s,u_1,u_2}$, where $W_s$ is the diagonal matrix  with diagonal elements being the quadrature weights in Lemma \ref{Lemma:fixed cubature-poly}, i.e., $\{w_{1,s},\dots,w_{|\Lambda|,s}\}.$

\begin{proposition}\label{proposition:quadrature-for-convolution}
Let $u,v\in[0,1/2]$,  $\hat{\phi}_k\stackrel{d}\sim k^{-2\gamma}$ with $\gamma>d/2$, $\phi, \varphi$ be SBFs satisfying \eqref{kernel-relation} with  $\alpha\gamma>d/2$.  If $\mathcal Q_{\Lambda,s}:=\{(w_{i,s},  x_{i}): w_{i,s}> 0
\hbox{~and~}   x_{i}\in \Lambda\}$ is a quadrature rule of degree $s$, then
   for any $\lambda>0$,  there holds
\begin{eqnarray}
 {WCE}^{\phi,\phi}_{\Lambda,\lambda,W,u,v}
 &\leq& C_7(1+q_\Lambda^{-1}s^{-1})^d\lambda^{-\max\{u,v\}}s^{-\gamma},\qquad \label{product-quadruare-111}\\
{WCE}_{\Lambda,\lambda,W,u}^{\varphi,\phi}
   &\leq &
 C_7(1+q_\Lambda^{-1}s^{-1})^d(\lambda^{-u}s^{-\gamma}+ s^{-\alpha\gamma}
   ),  \label{product-quadruare-222}
\end{eqnarray}
where $C_7$  is a constant  depending only on $\gamma, \alpha,u,v$ and $d$.  
\end{proposition}

Proposition \ref{proposition:quadrature-for-convolution} together with Proposition \ref{prop:equivalance-11}
shows that  quadrature rules are sufficient to derive tight bounds for operator differences, provided the sampling mechanism is appropriately designed. Recalling further Proposition \ref{Proposition:value-difference-deterministic} that small $\mathcal P_{D,\lambda,W,\beta}$ requires small $w_{\max}$, we then introduce two sampling and weighting schemes: spherical designs and equal weights, and $\tau$-quasi-uniform sets and D-type quadrature rules. The following two  corollaries   can be derived from  Proposition \ref{Proposition:value-difference-deterministic} and Proposition \ref{proposition:quadrature-for-convolution} directly.


\begin{corollary}\label{Cor:spherical-design-integral}
Let $0<\delta<1$, $u,v\in[0,1/2]$, $\hat{\phi}_k\stackrel{d}\sim k^{-2\gamma}$ with $\gamma>d/2$, and $\phi, \varphi$ be SBFs satisfying \eqref{kernel-relation} with  $\alpha\gamma>d/2$.
 Under Assumption \ref{Assumption:data}, if
   $\Lambda$ is a set of $s$-design with $s\leq c_\diamond|\Lambda|$ for some $c_\diamond>0$ and $w_i=\frac{1
}{|\Lambda|}$, then
 with confidence $1-\delta$, there holds
\begin{eqnarray*}
    \mathcal R_{\Lambda,\lambda,W,u,v} &\leq&
    C_7(1+q_\Lambda^{-1}s^{-1})^d\lambda^{-\max\{u,v\}}s^{-\gamma},\\
  \mathcal S_{\Lambda,\lambda,W,u}  &\leq &
    C_7(1+q_\Lambda^{-1}s^{-1})^d(\lambda^{-u}s^{-\gamma}+ s^{-\alpha\gamma}
   ),\\
    \mathcal P_{D,\lambda,W,\beta} 
    &\leq&
    C_7|\Lambda|^{-1/2}\lambda^{-d/(4 \gamma)}\log\frac6\delta,\qquad \forall 0\leq\beta\leq 1.
\end{eqnarray*}
If in addition $\Lambda$ is a $c_0$-tight $s$-design, then
\begin{eqnarray*}
 \mathcal R_{\Lambda,\lambda,W,u,v} \leq 
    C_8\lambda^{-\max\{u,v\}}s^{-\gamma},\qquad \mbox{and}\ 
  \mathcal S_{\Lambda,\lambda,W,u}   \leq 
    C_8(\lambda^{-u}s^{-\gamma}+ s^{-\alpha\gamma}
   ),   
\end{eqnarray*}
where $C_8$ is a constant depending only on $C_5,d,c_0$.
   \end{corollary}


\begin{corollary}\label{Cor:quasi-uniform-integral}
Let $0<\delta<1$, $u,v\in[0,1/2]$, $\hat{\phi}_k\stackrel{d}\sim k^{-2\gamma}$ with $\gamma>d/2$, and $\phi, \varphi$ be SBFs satisfying \eqref{kernel-relation} with  $\alpha\gamma>d/2$.  Under Assumption \ref{Assumption:data}, If 
  $\Lambda$ is $\tau$-quasi-uniform and $\mathcal Q_{\Lambda,s}:=\{(w_{i,s},  x^*_{i}): w_{i,s}> 0
\hbox{~and~}   x^*_{i}\in \Lambda\}$  is a D-type quadrature rule,
 then
 with confidence $1-\delta$, there holds
\begin{eqnarray*}
    \mathcal R_{\Lambda,\lambda,W,u,v} &\leq&
    C_9 \lambda^{-\max\{u,v\}}s^{-\gamma},\\
  \mathcal S_{\Lambda,\lambda,W,u}  &\leq &
    C_9 (\lambda^{-u}s^{-\gamma}+ s^{-\alpha\gamma}
   ),\\
    \mathcal P_{D,\lambda,W,\beta} 
    &\leq&
    C_9|\Lambda|^{-1/2}\lambda^{-d/(4\gamma)}\log\frac6\delta, \qquad \forall0\leq\beta\leq 1,
\end{eqnarray*}
 where $C_9$ is a constant depending only on $C_4,C_5,d,c_0$.
   \end{corollary}

\section{ Weighted Spectral Filter Algorithms and Operator Representations}\label{sec.spectral}

Different from KI that focuses on perfect fitting performance without considering the existence of noise, weighted spectral filter algorithms (WSFA), introduced in \cite{liu2024weighted}, devotes to introducing an additional filter  parameter to balance the fitting performance  and stability. The basic idea behind WSFA is to use spectral filters on   the matrix $\Psi_{\Lambda,W}$ to reduce the negative effect of small eigenvalues.

Let $g_\lambda:[0,\kappa]\rightarrow\mathbb R$ with $\lambda>0$ be the spectral filter function satisfying 
\begin{equation}\label{condition1}
\sup_{0<\sigma\le\kappa}|g_\lambda(\sigma)|\le \frac{b}{\lambda},
\qquad \sup_{0<\sigma\le \kappa}|g_\lambda(\sigma)\sigma|\le b,
\end{equation}
and
\begin{equation}\label{condition2}
\sup_{0<\sigma\le \kappa}|1-g_\lambda(\sigma)\sigma|\sigma^v\le
\tilde{C}_0\lambda^v, \qquad \forall\ 0\leq v\leq
{\nu}_g,
\end{equation}
where  $b>0$ is an absolute constant, $\nu_g>0$ is called as the qualification of the spectral filters and $\tilde{C}_0>0$ is an absolute  constant. Define
\begin{equation}\label{WSFA}
  f_{D,\lambda,W}=S^T_{D,W} g_\lambda(\Psi_{\Lambda,W})  {\bf y}_{D,W}= g_\lambda(L_{\phi,\Lambda,W}) S^T_{D,W} {\bf y}_{D,W},
\end{equation} 
where  $\Psi_{\Lambda,W}$ is given in \eqref{def.psi}, and the second equality is due to the fact that \cite[eqs(2.43)]{engl1996regularization} 
 \begin{equation}\label{exchange-adjoint}
     h(AA^T)A=Ah(A^TA) 
  \end{equation}
holds for any compact operator $A$  and piecewise continuous function $h(\cdot)$.

For any $f^*\in\mathcal N_\phi$,  it follows from \eqref{WSFA} and \eqref{Model1:fixed} that
\begin{eqnarray}\label{Analysis-decomposition}
      f_{D,\lambda,W}=g_\lambda(L_{\phi,D,W})S_{D,W}^T{\bf y}_{D,W}
      = \mathcal I_{D,g,\lambda}f^*
      +g_\lambda(L_{\phi,D,W})S_{D,W}^T\vec{\varepsilon}_{W},
\end{eqnarray}  
where $\mathcal I_{D,g,\lambda}:=g_\lambda(L_{\phi,D,W})L_{\phi,D,W}$
and
$\vec{\varepsilon}_{W}:=(\sqrt{w_1}\varepsilon_1,\dots,\sqrt{w_{|\Lambda|}}\varepsilon_{|\Lambda|})^T$. Similar as $\mathcal I_D$ in \eqref{clean-KI}, 
$\mathcal I_{D,g,\lambda}:\mathcal H_K\rightarrow\mathcal H_K$ can be regarded as a quasi-interpolation operator. On one hand, \eqref{condition1} shows that 
\begin{equation}\label{lebesgue}
    \|\mathcal I_{D,g,\lambda}\|_{\mathcal N_\phi\rightarrow\mathcal N_\phi}  \leq 1,
\end{equation}
presenting the stability of the quasi-interpolation. Different from KI in which it is difficult to present a tight bound \cite{narcowich1998stability} of $\sigma_{1}(\Psi_{\Lambda,W}^{-1})$, \eqref{condition1} also yields $\sigma_1(g_{\lambda}(\Psi_{\Lambda,W}))\leq \lambda^{-1}$. In a nutshell, \eqref{condition1} presents both Lebesgue constant type stability \cite{hangelbroek2010kernel} and minimal eigenvalue type stability \cite{narcowich1998stability} for WSFA.  On the other hand, \eqref{condition2} quantifies the fitting performance by encoding the regularity of $f^*$ in the approximation process.  In particular, for any $f^*\in\mathcal N_\phi$, we get from $\|f\|_{L^2}=\|L_\phi^{1/2}f\|_{\phi}$ that
\begin{eqnarray*} 
    &&\|f^*-\mathcal I_{D,g,\lambda}f^*\|_{L^2}
    =\|L_\phi^{1/2}(g_\lambda(L_{\phi,D,W})L_{\phi,D,W}-I)f^*\|_\phi\nonumber\\
    &\leq& \mathcal Q_{\Lambda,\lambda,W} \|(g_\lambda(L_{\phi,D,W})L_{\phi,D,W}-I)L_{\phi,D,W}^{1/2}\|\|f^*\|_\phi \\
     &\leq &
   \mathcal Q_{\Lambda,\lambda,W}\sup_{0<\sigma\le \kappa}|1-g_\lambda(\sigma)\sigma|\sigma^{1/2}\|f^*\|_\phi 
    \leq 
  \tilde{C}_0\lambda^{1/2} \mathcal Q_{\Lambda,\lambda,W} \|f^*\|_\phi,
\end{eqnarray*}
where $\mathcal Q_{\Lambda,\lambda,W}$ is given in \eqref{def:Q}.   We then present the second assumption concerning the filter $g_\lambda$.
\begin{assumption}\label{Assumption:filter}
Assume that $g_\lambda:[0,\kappa]$ satisfying \eqref{condition1} and \eqref{condition2} with qualification $\nu_g$.
\end{assumption}

There are numerous filter functions satisfying Assumption \ref{Assumption:filter}. For instance, Tikhonov regularization \eqref{KRR} is a special type of WSFA with qualification $\nu_g=1$.  
  We refer the readers to Table \ref{Tab:comparison} for more details on the spectral filters for Tikhonov regularization, Landweber iteration, $\nu$-method, spectral cut-off and iterated Tikhonov. Their numerical implementation as well as their computational burden, storage requirements can be easily deduced from  \cite{gerfo2008spectral}.  Theoretical verifications for WSFA were firstly provided in \cite{liu2024weighted} when $D=\{(x_i,y_i)\}_{i=1}^{|D|}$ satisfying \eqref{Model1:fixed} for $f^*\in\mathcal N_\phi$ and mean-zero random bounded noise, i.e., $E[\varepsilon_i]=0$ and $|\varepsilon|\leq M_0$ for some $M_0>0$. d
  \begin{table}[H]
    \begin{center}
	\begin{tabular}{|c|c|c|c|c|}
		\hline
		Algorithm & Storage  & Computation & Filter function $g_{\lambda}(\sigma)$ & $ \nu_g$ \\
		\hline
		Tikhonov & $O(|\Lambda|^{2})$ & $O(|\Lambda|^{3})$ & $g_{\lambda}(\sigma)=\frac{1}{\sigma+\lambda}$ & 1 \\
		\hline
		Landweber & $O(|\Lambda|^{2})$ & $O(|\Lambda|^{2})$ & $g_{t}(\sigma)=\tau \sum_{i=0}^{t-1}(1-\tau \sigma)^{i}$ & $\infty$ \\
		\hline
		$\nu$-method & $O(|\Lambda|^{2})$ & $O(|\Lambda|^{3})$ & $g_{t}(\sigma)=p_{t}(\sigma)$ & { $\infty$} \\
		\hline
		cut-off & $O(|\Lambda|^{2})$ & $O(|\Lambda|^{3})$ & $g_{\lambda}(\sigma)=\{\begin{array}{cc}\frac{1}{\sigma} & \sigma \geq \lambda \\ 0 & \sigma<\lambda\end{array}$ & $\infty$ \\
		\hline
		iterated & $O(|\Lambda|^{2})$ & $O(|\Lambda|^{3})$ & $g_{\lambda}(\sigma)=\frac{(\sigma+\lambda)^{v}-\lambda^{v}}{\sigma(\sigma+\lambda)^{v}} $ & {  $v$} \\
		\hline
    \end{tabular}
    \end{center}
      \caption{Spectral filter functions and corresponding properties}\label{Tab:comparison}
    \end{table}

Different from \cite{hesse2017radial,liu2024weighted}
we are also  concerned with noisy data fitting problems satisfying \eqref{Assumption:data}   with $0<\alpha<1$. 
The unboundedness of random noise reflects the existence of outliers in the sampling process. The existence of unbounded noise   and  NSB has not been considered  in the literature \cite{narcowich2002scattered,levesley2005approximation,le2006continuous,narcowich2007direct,le2008localized,hangelbroek2010kernel,hangelbroek2011kernel,hesse2017radial,Feng2021radial,lin2021subsampling,liu2024weighted} and requires novel operator differences such as $\mathcal L_{\phi}, \mathcal L_{\Lambda,\phi,W}$, and operator representations in the following theorem.
 
\begin{theorem}\label{Theorem:operator-representation}
Let $\hat{\phi}_k\stackrel{d}\sim k^{-2\gamma}$ with $\gamma>d/2$, $\phi$, $\varphi$ satisfy \eqref{kernel-relation} for $\alpha\gamma\geq  \beta/2$ and  $g_\lambda$ satisfies \eqref{condition1} and \eqref{condition2}. Under Assumptions \ref{Assumption:data} and \ref{Assumption:filter} and $\mathcal R_{\Lambda,\lambda,W,1/2,1/2}\leq \tilde{c}<1/2$, if  $\alpha<1$,
  then
\begin{eqnarray}\label{Approximation-operator-represen}
   \|J_{\phi}f^\diamond_{D,W,\lambda}-f^*\|_{L^2}
      \leq  
      \bar{C}_1 (\overbrace{ \mathcal P_{D,\lambda,W,0}+ 
       \mathcal R_{\Lambda,\lambda,W,\frac12,\frac12} 
      + 
      \mathcal S_{\Lambda,\lambda,W,\frac12}}^{\mbox{stability error}}
       + 
      \overbrace{\lambda^{\min\{\frac{\alpha}{2},\nu_g\}}
   }^{\mbox{fitting error}}). 
\end{eqnarray}
If  $\alpha>1$, then
\begin{eqnarray}\label{Approximation-operator-represen-large}
   \|f^\diamond_{D,W,\lambda}-f^*\|_{\psi}
     \leq  
      \bar{C}_1 (\overbrace{   \lambda^{-\frac{\beta}{2}}\mathcal P_{D,\lambda,W,1} }^{\mbox{stability error}} 
       + 
      \overbrace{ \lambda^{-\beta/2}(\lambda^{\min\{\frac{\alpha}{2},\nu_g\}}+\lambda^{\min\{\frac{1}{2},\nu_g\}}\mathcal R_{\Lambda,\lambda,W,0,0}\mathbb I_{\alpha>2})}^{\mbox{fitting error}}), 
\end{eqnarray}
where $\bar{C}_1$ is a constant depending only on $\|f^*\|_\varphi,b,\kappa,\tilde{c},\alpha$ and  $\mathbb I_{\mathcal A}$ denotes the indicator of the event $\mathcal A$ and 
\end{theorem}




In Theorem \ref{Theorem:operator-representation}, 
besides $\mathcal P_{D,\lambda,W,0}$, additional two terms concerning operator differences are introduced due to $f^*\notin\mathcal N_\phi$. Indeed, we first find an $f_\phi\in\mathcal N_\phi$ to approximate   $f^*$ well and then treat the remnants of the approximation  $f^*-f_\phi$   and $f_\phi$ respectively as the artificial noise  and the new target function.  The terms $\mathcal R_{\Lambda,\lambda,W,1/2,1/2}$ and  $\mathcal S_{\Lambda,\lambda,W,1/2}$ are used to quantify the artificial noise $f^*-f_\phi$. If $\alpha\geq 1$ in Theorem \ref{Theorem:operator-representation},  $\mathcal R_{\Lambda,\lambda,W,1/2,1/2}$ and  $\mathcal S_{\Lambda,\lambda,W,1/2}$ are removable and $\mathcal P_{D,\lambda,W,1}$ is sufficient to describe the stability. 
We compare our stability results and related work in \cref{Tab:comparison-stability}. 
 \begin{table}[H]
    \begin{center}
	\begin{tabular}{|c|c|c|c| c|}
		\hline
		Reference & measurement   & noise &  $f^*$ &algorithm\\
		\hline
		\cite{narcowich1998stability} & $\sigma_1(\Phi_D)$
  & NON & $f^*\in \mathcal N_\phi$  & KI     \\
		\hline
		\cite{hangelbroek2010kernel} &   $\|\mathcal I_D\|_{L^\infty\rightarrow L^\infty}$  & NON   & $f^*\in \mathcal N_\phi$  & KI\\
		\hline
		\cite{hangelbroek2011kernel} & $\|\mathcal I_D\|_{L^2 \rightarrow L^\infty}$ & NON & $f^*\in \mathcal N_\phi$ &KI\\
  \hline
		\cite{hesse2017radial} &   $\mathcal P_{D,\lambda,W,1}$  & small  & $f^*\in \mathcal N_\phi$  & Tikhonov\\
  \hline
  \cite{Feng2021radial} & $\mathcal P_{D,\lambda,W,1}$ & bounded & $f^*\in \mathcal N_\phi$ &Tikhonov\\
		\hline
  \cite{liu2024weighted}   & $\mathcal P_{D,\lambda,W,1}$ & bounded & $f^*\in \mathcal N_\phi$  & WSFA\\
		\hline
  This paper   & Hybrid & unbounded & $f^*\notin \mathcal N_\phi$ & WSFA\\
		\hline
    \end{tabular}
    \end{center}
      \caption{Comparisons of different measurements of stability }\label{Tab:comparison-stability}
    \end{table}

 We then show   that spectral filters satisfying Assumption \ref{Assumption:filter} do not destroy the fitting performance of KI, provided the filter parameter $\lambda$ is appropriately selected. 
Different from the stability error, the fitting error presented in \cref{Theorem:operator-representation} increases with respect to the filter parameter $\lambda$, showing that small $\lambda$ yields small fitting error. In particular, if we select $g_\lambda$ satisfying \eqref{condition2} with $\nu_g\geq \alpha/2$, $\Lambda$ is quasi-uniformly sampled and $\{w_{i,s}\}_{i=1}^{|\Lambda|}$ is  D-type quadrature weights, then setting $\lambda=c_{**}s^{-2\gamma}$ follows   
$\mathcal R_{\Lambda,\lambda,W_s,1/2,1/2}\leq \tilde{c}<1/2$ according to Corollary \ref{Cor:quasi-uniform-integral}, and 
the approximation error in \eqref{Approximation-operator-represen} is
\begin{equation}\label{fitting-error-3}
    \|J_{\phi}f^\diamond_{D,\lambda,W_s}-f^*\|_{L^2}
     \leq 
    2\bar{C}_2 
    |\Lambda|^{-\alpha\gamma/d},
\end{equation}
which is the same as that for KI presented in \cref{Lemma:Narcowich-interpolation} under the same conditions. 

Fitting performance of WSFA have already  been studied in \cite{hesse2017radial,Feng2021radial,liu2024weighted}. In particular, \cite{hesse2017radial} considered Tikhonov regularization and fitting error of   order $|\Lambda|^{-\gamma/d}$ was derived for $w_i=1/|\Lambda|$, $i=1,\dots,|\Lambda|$, $\alpha=1$ and $\beta=0$;  
\cite{Feng2021radial} adopted weighted Tikhonov regularization and fitting error of   order $|\Lambda|^{-\alpha\gamma/ (2d) }$ was deduced for  $1\leq \alpha\leq 2$; \cite{liu2024weighted} focused on WSFA and fitting error of   order $|\Lambda|^{-\alpha\gamma/d}$ for $0\leq\beta\leq 1$ and $\alpha\geq 1$. We compare our results and these interesting work in \cref{Tab:comparison-11}. 
 \begin{table}[H]
    \begin{center}
	\begin{tabular}{|c|c|c|c|c|c|}
		\hline
		Reference & NSB   & saturation  & adaptivity &  optimality  &algorithm\\
  \hline
		\cite{narcowich2002scattered} & \XSolidBrush 
  & \Checkmark  & \XSolidBrush & \Checkmark &KI \\ 
		\hline
  	\cite{narcowich2007direct} & \Checkmark 
  & \XSolidBrush & \Checkmark  & \Checkmark &KI \\ 
		\hline
		\cite{hesse2017radial} & \XSolidBrush 
  & \XSolidBrush &\XSolidBrush  & \Checkmark &Tikhonov \\ 
		\hline
		\cite{Feng2021radial} & \XSolidBrush  & \XSolidBrush &\XSolidBrush    &\XSolidBrush &Tikhonov\\
		\hline
		\cite{liu2024weighted} & \XSolidBrush & \Checkmark & \Checkmark & \Checkmark &WSFA\\
		\hline
  This paper & \Checkmark & \Checkmark &\Checkmark &\Checkmark &WSFA \\
		\hline
    \end{tabular}
    \end{center}
      \caption{Comparisons of fitting peformance  among different results   on    conquering the native-space-barrier (NSB), circumventing the saturation (saturation), achieving optimal rates (optimality) and 
     adapting  to different metrics  (adaptivity)
      }\label{Tab:comparison-11}
    \end{table}

Combining Theorem \ref{Theorem:operator-representation} with Corollary \ref{Cor:quasi-uniform-integral} and Corollary \ref{Cor:spherical-design-integral}, we can derive explicit approximation rates of WSFA and show that the derived rates are at least not worse than the existing results in \cite{Feng2021radial,liu2024weighted}.

\begin{corollary}\label{Corollary:app-rate-quasi}
 Let $\delta\in(0,1)$, $\hat{\phi}_k\stackrel{d}\sim k^{-2\gamma}$ with $\gamma>d/2$ and $\phi$,  $\varphi$ satisfy \eqref{kernel-relation} for   $\alpha\gamma>d/2$ and $\alpha\leq 1$. Suppose that  $\Lambda=\{x_i\}_{i=1}^{|\Lambda|}$ is $\tau$-quasi uniform,  $\mathcal Q_{\Lambda,s}=\{(w_{i,s},  x_i):
\hbox{~and~}   x_i\in \Lambda\}$ is a D-type quadrature rule with coefficient $c_*$, and $s=c_\diamond |\Lambda|^{1/d}$. Under Assumptions \ref{Assumption:data} and \ref{Assumption:filter},
If  $\nu_g\geq \frac{\alpha}{2}$ and $\lambda\stackrel{d}\sim|\Lambda|^{-\frac{2\gamma}{2\gamma\alpha+d}}$,
 then   with confidence $1-\delta$, 
there holds
 \begin{equation}\label{approxiamtion-error-random}
     \|J_{\phi}f_{D,W_s,\lambda}-f^*\|_{L^2}  
   \leq \bar{C}_3|\Lambda|^{-\frac{\gamma\alpha}{2\gamma\alpha+d}}\log\frac{6}{\delta}, 
 \end{equation}
where $\bar{C}_3$ is a constant independent of $|\Lambda|,\lambda,\delta$.
\end{corollary}

The derived approximation rates in \eqref{approxiamtion-error-random} are optimal in the sense that \cite{lin2021subsampling} under the conditions of Corollary \ref{Corollary:app-rate-quasi}, there exists a sub-Gaussian distribution $\rho^*$ of $\varepsilon$ and an $f^*_{bad}\in\mathcal N_{\varphi}$  such that
\begin{eqnarray}\label{Lower-bound}
     P_{\rho^*}\left[\|f_{D,\lambda,W_s}-f^*_{bad}\|_{L^2}\geq \bar{C}_4|\Lambda|^{-\frac{\gamma\alpha}{2\gamma\alpha+d}}\right]
    \geq
    \frac{1}{4},
\end{eqnarray}
where $  P_{\rho^*}$ denotes the probability with respect to the distribution $\rho^*$, and $\bar{C}_4$ is a constant  independent of $|\Lambda|$, $\delta, s,\lambda$.
Similar corollary for $\alpha\geq 1$ can also be easily deduced as follows.

\begin{corollary}\label{Corollary:large-optimal}
    Let $\delta\in(0,1)$, $\hat{\phi}_k\stackrel{d}\sim k^{-2\gamma}$ with $\gamma>d/2$ and $\phi$, $\psi$ $\varphi$ satisfy \eqref{kernel-relation} for   $0\leq\beta\leq 1$ and $\alpha\geq 1$.  Suppose that  $\Lambda=\{x_i\}_{i=1}^{|\Lambda|}$ is $\tau$-quasi uniform,  $\mathcal Q_{\Lambda,s}=\{(w_{i,s},  x_i):
\hbox{~and~}   x_i\in \Lambda\}$ is a D-type quadrature rule  of degree   $s=c_\diamond |\Lambda|^{1/d}$ and with coefficient $c_*$. If  Assumption \ref{Assumption:data} and Assumption \ref{Assumption:filter} hold with $\nu_g\geq\frac{\alpha}{2}$ and $\lambda\stackrel{d}\sim|\Lambda|^{-\frac{2\gamma}{2\gamma\alpha+d}}$, then
\begin{eqnarray*}
   \|f^\diamond_{D,W_s,\lambda}-f^*\|_{\psi}
     \leq   \bar{C}_3' |\Lambda|^{- \frac{2(\alpha-\beta)\gamma}{2\alpha\gamma+d}}\log\frac{6}{\delta},
\end{eqnarray*}
where $\bar{C}_3'$ is a constant independent of $|\Lambda|,\lambda,\delta$.
\end{corollary}

\section{Lepskii-type Principle based on Operator Representations}\label{Sec:parameter}

Different from KI, WSFA introduces a tunable parameter to improve  its stability without sacrificing  its fitting performance as discussed in Corollary \ref{Corollary:app-rate-quasi}. In this section, we utilize the   operator representations in Theorem \ref{Theorem:operator-representation} and operator differences in Corollary \ref{Cor:quasi-uniform-integral}  to provide a Lepskii-type principle  to determine the filter  parameter $\lambda$ with theoretical guarantees. Our idea is motivated by \cite{de2010adaptive,lu2020balancing,blanchard2019lepskii,lin2024adaptive} for random sample, but the detailed designations and proof skills are totally different due to the exclusive integral operator approach for spherical data.
Without loss of generality, we only study the well-specified case, i.e.,$\alpha>1$, for quasi-uniform data with D-type quadrature rules. Other cases  can be done by using the same method.
 Our basic idea of the Lepskii-type principle is to use   computational quantities concerning stability to bound
the difference between two successive estimators. 
To this end,  we present the following proposition.

\begin{proposition}\label{Proposition:stopping-rule}
 Let $\hat{\phi}_k\stackrel{d}\sim k^{-2\gamma}$ with $\gamma>d/2$ and $\phi$,  $\varphi$ satisfy \eqref{kernel-relation} for   $\alpha\geq1$. 
For any $\lambda,\lambda'$ satisfying $\mathcal R_{\Lambda,\lambda,W,1/2,1/2}, \mathcal R_{\Lambda,\lambda',W,1/2,1/2} \leq \tilde{c}<1/2$ and $\lambda>\lambda'$, if Assumption \ref{Assumption:data} and Assumption \ref{Assumption:filter} hold with $\nu_g\geq \frac{\alpha}{2}$, 
 then    
\begin{eqnarray}\label{stop-1111}
   &&\left\|(  L_{\phi,\Lambda,W}+\lambda I)^{1/2} ( f_{D,W,\lambda}-f_{D,W,\lambda'})\right\|_\phi \nonumber\\
     &\leq&
      \bar{C}_5 (\lambda^{\min\{\frac{\alpha}{2},\nu_g\}}+\lambda^{\min\{\frac{1}{2},\nu_g\}}\mathcal R_{\Lambda,\lambda,W,0,0}\mathbb I_{\alpha>2} )
      +
      8b(1-\tilde{c})^{-3/2} \mathcal P_{D,\lambda',W,1},
 \end{eqnarray}
where $\bar{C}_5:=4\tilde{C}_0 (1-\tilde{c})^{-5/2}\max\{1,(\alpha-1) \kappa^{\alpha-2}\}$.
\end{proposition}

Comparing \eqref{stop-1111} with \eqref{Approximation-operator-represen-large}, we get that the trends of fitting errors and stability errors are the same for $\left\|(  L_{\phi,D,W}+\lambda I)^{1/2} ( f_{D,W,\lambda}-f_{D,W,\lambda'})\right\|_\phi$ and $\|f^\diamond_{D,W,\lambda}-f^*\|_{\psi}$. The only  difference is that $ \left\|(  L_{\phi,D,W}+\lambda I)^{1/2} ( f_{D,W,\lambda}-f_{D,W,\lambda'})\right\|_\phi$ is implementable but $\|f^\diamond_{D,W,\lambda}-f^*\|_{\psi}$ cannot be computed. Indeed,
 for any   $f=\sum_{i=1}^{|\lambda|}a_i\phi_{x_i}$, there holds
\begin{eqnarray*}
  \|(L_{\phi,\Lambda,W}+\lambda I)^{1/2}f\|^2_\phi
  = 
\langle(L_{\phi,\Lambda,W}+\lambda I)f,f\rangle_\phi
=\langle  L_{\phi,\Lambda,W}f,f\rangle_\phi
+
\lambda \langle  f,f\rangle_\phi,
\end{eqnarray*}
which together with $L_{\phi,\Lambda,W}=S_{D,W}^TS_{D,W}$ and \eqref{Reproducing-property} follows
\begin{eqnarray*}
     \|(L_{\phi,\Lambda,W}+\lambda I)^{1/2}f\|^2_\phi
    &=&
    \langle S_{D,W}f,S_{D,W}f\rangle_{\mathbb R^d}+\sum_{i=1}^{|\Lambda|}\sum_{i'=1}^{|\Lambda|}a_ia_{i'}\phi(x_i,x_{i'})\\
    &=&
\sum_{i=1}^{|\Lambda|}w_i(f(x_i))^2+\sum_{i=1}^{|\Lambda|}\sum_{i'=1}^{|\Lambda|}a_ia_{i'}\phi(x_i,x_{i'}).
\end{eqnarray*}
If   $\Lambda$ is $\tau$-quasi-uniform and  $\mathcal Q_{\Lambda,s}:=\{(w_{i,s},  x^*_{i}): w_{i,s}> 0
\hbox{~and~}   x^*_{i}\in \Lambda\}$ is a D-type quadrature rule  of order $s=c_\diamond |\Lambda|^{1/d}$ and coefficient $c_*>0$. Recalling Proposition \ref{Proposition:value-difference-deterministic}, if $\varepsilon_i$ is a random sub-Gaussian noise, then with confidence $1-\delta$ for $\delta\in(0,1)$, there holds
$$ 
   \mathcal P_{D,\lambda,W_s,1} 
     \leq     
     2c_*^2C_6\lambda^{-d/(4\gamma)} |\Lambda|^{-1/2}\log\frac6\delta, 
$$
showing that the stability error is also implementable since $\gamma$ is known provided the kernel is given \cite{narcowich2007approximation}. Noting further the best $\lambda$ is achieved when the fitting error is comparable with the stability error and is increasing with respect to $\lambda$, we then can determine $\lambda$ by comparing $\left\|(  L_{\phi,\Lambda,W_s}+\lambda I)^{1/2} ( f_{D,W_s,\lambda}-f_{D,W_s,\lambda'})\right\|_\phi $ and $ C_6\lambda^{-d/(4\gamma)} \left(\sum_{i=1}^{|\Lambda|}w_{i,s}^2+w_{\max,s}^2\right)^{1/2} \log\frac6\delta$. With these, we are in a position to give the Lepskii-type principle for WSFA.

For $0<q<1$ and  $q_0>0$, define
\begin{equation}\label{Def.k}
    K_q:=\log_q\frac{\sqrt{3C_9s^{-\gamma}}}{q_0}
\end{equation}
with $C_9$ given in Corollary \ref{Cor:quasi-uniform-integral}. Then for any  $\lambda_k=q_0q^k$ with $k\leq K_q$, it follows from Corollary \ref{Cor:quasi-uniform-integral} that  
\begin{equation}\label{bound-rrrrrr}
    \mathcal R_{\Lambda,\lambda_k,W_s,1/2,1/2}.
\end{equation}
Let $\hat{k}$ be the first (or largest) integer of $K_q,K_q-1,\dots,2,1$ satisfying
\begin{equation}\label{Lepskii-1}
    \left\|(  L_{\phi,\Lambda,W_s}+\lambda_k I)^{1/2} ( f_{D,W_s,\lambda_k}-f_{D,W_s,\lambda_{k-1}})\right\|_\phi
    \leq C_{LP}\lambda^{-d/(4\gamma)} |\Lambda|^{-1/2}\log\frac6\delta,
\end{equation}
 where  $C_{LP}:=32\bar{C}_5c_*^2C_6b(3/2)^{-3/2}$ and
 $C_6$ is specified in the proof of Proposition \ref{Proposition:value-difference-deterministic}. If there is not any $k\in[1,K_q]$ satisfying \eqref{Lepskii-1}, set $\hat{k}=K_q$. We then derive the following theorem to show the feasibility of the proposed parameter selection strategy.
\begin{theorem}\label{Theorem:Lepskii}
Let $\delta\in(0,1). $ Under  conditions of Corollary \ref{Corollary:large-optimal},  with confidence $1-\delta$, 
there holds
 \begin{equation}\label{approxiamtion-Lepskii}
     \|f_{D,W_s, {\lambda}_{\hat{k}}}-f^*\|_{\psi}  
   \leq \bar{C}_6|\Lambda|^{-\frac{(\alpha-\beta)\gamma }{2\gamma\alpha+d}} \log\frac{3}{\delta}, 
 \end{equation}
where $\bar{C}_6$ is a constant independent of $|\Lambda|,\delta$.
\end{theorem}

Comparing Theorem \ref{Theorem:Lepskii} with Corollary \ref{Corollary:large-optimal}, we show that the proposed Lepskii-type principle in \eqref{Lepskii-1} makes WSFA achieve
 the optimal approximation rates   by noticing \eqref{Lower-bound}. 
Parameter selection of Tikhonov regularization for spherical data fitting has already been considered in \cite{hesse2017radial}, where four types of strategies such as  the well known Morozov discrepancy principle were proposed and corresponding approximation rates were derived to demonstrate the feasibility of these strategies. However, it should be highlighted that the noise considered in \cite{hesse2017radial} is extremely small and the error analysis were only carried out in the $L^2$-metric.  In our recent work \cite{liu2024weighted}, a weighted cross-validation approach was developed to determine the filter parameter of WSFA. There are mainly three differences between our results and those in \cite{liu2024weighted}. At first, the weighted cross-validation in \cite{liu2024weighted} requires to divide the sample into training and validation sets while our proposed Lepskii principle can fully use the whole data. Due to the lack of concentration inequalities, the weighted cross-validation in \cite{liu2024weighted} needs to use the validation sets to generate a quadrature rule of D-type, which naturally requires the validation set  to be quasi-uniform and needs additional computation.
Then, the weighted cross-validation in \cite{liu2024weighted} is only available to $L^2$ norm while our results also hold for Sobolev norms.  Finally, the derived approximation errors in \cite{liu2024weighted} are built upon the assumption $\gamma>3d/2$ while  our results only requires $\gamma>d/2$, theoretically exhibiting the advantage of the Lepskii-type principle over weighted cross-validation.

\section{Divide-and-Conquer Scheme for Computation-Reduction}\label{sec:Dis}
As shown in Table \ref{Tab:comparison},  WSFA  requires at least $\mathcal O(|\Lambda|^2)$  and generally $\mathcal O(|\Lambda|^3)$ complexities in computation and storage. In this section, we introduce a divide-and-conquer (D-C) scheme to reduce complexities of   WSFA. The algorithm mainly contains three steps. 

 \begin{itemize}
     \item 
 {\it Step 1. Data set decomposition:} For a fixed $1 \le J \le |\Lambda|$, decompose the data set $D$ into $J$  disjoint subsets $D_j:=\{(x_{i,j},y_{i,j}\}_{i=1}^{|\Lambda_j|}$ so that $D=\bigcup_{j=1}^JD_j$, $D_j\cap D_{j'}=\varnothing$ for $j\neq j'$, and the input set $\Lambda_j$ of $D_j$ is $\tau_j$-quasi-uniform for some $\tau_j\geq 2$ depending only on $\tau,d$. The detailed implementation for the decomposition of a set of scattered data to $J$ $\tau_j$-quasi-uniform set   can be found in \cite{lin2023dis}. For each $\tau$-quasi-uniform subset, generate a D-type quadrature rule $\mathcal Q_{\Lambda_j,s_j}=\{(w_{i,s_j},  x_{i,j}):
\hbox{~and~}   x_{i,j}\in \Lambda_j\}$  of order $s_j=c_\diamond|\Lambda_j|^{1/d}$ and with coefficient  $c_*$.

\item  {\it Step 2. Local processing:} On each subset $D_j$,  implement \eqref{WSFA} and obtain a local estimator
\begin{equation}\label{local estimator}
    f_{D_j^*,W_{s_j},\lambda_j}=S^T_{\Lambda_j,W_{s_j}} g_{\lambda_j}(\Psi_{\Lambda_j^*,s_j})  {\bf y}_{D_j,W_{s_j}}  
\end{equation}
with $\lambda_j>0$.

\item {\it Step 3. Synthesization:} Build the global estimator by synthesizing local estimators via weighted averaging
\begin{equation}\label{global estimator}
    \bar{f}_{D,W_{\vec s},\vec{\lambda}}=\sum_{j=1}^{J}  \frac{|\Lambda_j|}{|\Lambda|}f_{D_j,W_{s_j},\lambda_j}.
\end{equation}
 \end{itemize}

Utilizing the D-C scheme, it is easy  to see that the 
complexities of WSFA are reduced from $\mathcal O(|\Lambda|^2)$ and $\mathcal O(|\Lambda|^3)$ to $\mathcal O(\sum_{j=1}^J|\Lambda_j|^2)$ and $\mathcal O(\sum_{j=1}^J|\Lambda_j|^3)$ in storage and computation, respectively. If  $|\Lambda_1|\sim\dots\sim |\Lambda_J|$, the above quantities become $\mathcal O(|\Lambda|^2/J)$ and $\mathcal O(|\Lambda|^3/J^2)$, showing that large $J$ brings additional benefits in computation.  In the following theorem, we show that the proposed D-C scheme does not degrade the approximation performance of WSFA, provided $J$ is not so large.


\begin{theorem}\label{Theorem:DS-optimal}
Let $\delta\in(0,1),$ $\hat{\phi}_k\stackrel{d}\sim k^{-2\gamma}$ with $\gamma>d/2$ and $\phi$  $\varphi$ satisfy \eqref{kernel-relation} for   $\alpha>1$. For any $j=1,\dots,J$, suppose that  $\Lambda_j=\{  x_{i,j}\}_{i=1}^{|\Lambda_j|}$ is $\tau_j$-quasi uniform, and $\mathcal Q_{\Lambda_j,s_j}=\{(w_{i,s_j},  x_{i,j}):
\hbox{~and~}   x_{i,j}\in \Lambda_j\}$ is a D-type quadrature rule with coefficient $c_*$ and $s=c_\diamond |\Lambda_j|^{1/d}$. If Assumption \ref{Assumption:data} and Assumption \ref{Assumption:filter} hold with $\nu_g\geq \frac{\alpha}{2}$, $\lambda\sim|\Lambda|^{-\frac{2\gamma}{2\gamma\alpha+d}}$, $|\Lambda_1|\sim\dots\sim|\Lambda_j|$, and
\begin{equation}\label{restrtion-on-J}
    J\leq \bar{C}_7|\Lambda|^{\frac{2\gamma\alpha}{2\gamma\alpha+d}},
\end{equation}
 then for any $0<\delta<1$,    with confidence $1-\delta$, 
there holds
\begin{equation}\label{Approximation-rate-distributed}
  E[\| \overline{f}_{D,W_{\vec{s}},\vec{\lambda}}-f^*\|_{\psi}^2] 
       \leq 
         \bar{C}_8|\Lambda|^{-\frac{2(\alpha-\beta)\gamma}{2\gamma\alpha+d}}, 
 \end{equation}
where $\bar{C}_7,\bar{C}_8$ are constants independent of $|\Lambda|,|\Lambda_j|,J,\lambda,s$.
\end{theorem}
 
  Comparing Theorem \ref{Theorem:DS-optimal} with Corollary \ref{Corollary:app-rate-quasi}, we find that  D-C-WSFA possesses the same optimal approximation rates as WSFA, showing the feasibility and efficiency of the  D-C scheme. Furthermore, the filter parameters to achieve the optimal approximation rates of WSFA and D-C-WSFA are the same. Therefore, we can use the proposed Lepskii-type principle  \eqref{Lepskii-1} to determine the parameter of D-C-WSFA.
  Recalling that there are only $|\Lambda_j|$ samples in the $j$-th data subset, such a selection of $\lambda_j$ enables the local estimator $f_{D_j^*,W_{s_j},\lambda_j}$ to be overfitted on the data, i.e., the fitting error is much smaller than the stability error, which can be derived from Theorem \ref{Theorem:operator-representation}   by setting $\lambda_j\sim|\Lambda|^{-\frac{2\gamma}{2\gamma\alpha+d}}$ directly. In fact, the best selection of filter parameter for the local estimator should be $\lambda_j\sim |\Lambda_j|^{-\frac{2\gamma}{2\gamma\alpha+d}}$. Fortunately, the overfitting of local estimates can be alleviated by weighted average, making the approximation rates of the global estimator similar as those for WSFA. The condition \eqref{restrtion-on-J} shows that the performance of weighted average in enhancing the stability is limited, if the whole data set are divided into   too many data subsets. In fact, for  $|\Lambda_1|\sim\dots\sim|\Lambda_j|$ and $J= \tilde{C}_1|\Lambda|^{\frac{2\gamma\alpha}{2\gamma\alpha+d}}$, there are totally $\mathcal O(|\Lambda|^{\frac{d}{2\gamma\alpha+d}})$ samples in each data subset. The optimal approximation error, even for noise-free data, is of order $|\Lambda|^{-\frac{\gamma\alpha}{2\gamma\alpha+d}}$ according to \cref{Lemma:Narcowich-interpolation}, implying the optimality of the restriction on $J$.

\section{Proofs}\label{Sec:proof}
In this section, we present proofs of our theoretical  results.

\subsection{Proofs of auxiliary results in Sections \ref{sec.SBF} and \ref{Sec:integral-operator}} 

We at first prove Lemma \ref{Lemma:source-condition}.

\begin{proof}[Proof of Lemma \ref{Lemma:source-condition}]
Due to \eqref{relation11111} and \eqref{kernel-relation}, we have 
\begin{equation}\label{population-inetral-repre}
    \eta( \mathcal L_{\phi,\psi})g(x)=\sum_{k=0}^\infty\eta\left(\frac{\hat{\phi}_k}{\hat{\psi}_k}\right) \sum_{\ell=1}^{Z(d,k)}\hat{g}_{k,\ell}Y_{k,\ell}(x)
    =
     \sum_{k=0}^\infty\eta\left(\hat{\phi}_k^{1-\beta}\right) \sum_{\ell=1}^{Z(d,k)}\hat{g}_{k,\ell}Y_{k,\ell}(x), 
\end{equation}
implying
$$
   \|\eta( \mathcal L_{\phi,\psi})g\|_\psi=
   \sum_{k=0}^\infty \psi_{k}^{-1}
  \left(\eta\left(\hat{\phi}_k^{1-\beta}\right)\right)^2\sum_{\ell=1}^{Z(d,k)}\hat{g}_{k,j}^2.
$$
Then, \eqref{population-inetral-repre} with $\eta(t)=t^{1-\beta}$ and $\beta=0$ yields
$$
    \mathcal L_\phi^{1-\beta} g(x)= \sum_{k=0}^\infty \hat{\phi}_k^{1-\beta} \sum_{\ell=1}^{Z(d,k)}\hat{g}_{k,\ell}Y_{k,\ell}(x).
$$
Comparing the above equality with \eqref{population-inetral-repre} with $\eta(t)=t$, we then have
$\mathcal L_{\phi,\psi}g=\mathcal L_{\phi}^{1-\beta}g$ for any $g\in\mathcal N_\psi$. The same approach as above can derive $ L_{\phi,\psi}=  L_{\phi}^{1-\beta}$ and verifies \eqref{operator-relation-112}. 
For $f\in\mathcal N_\phi$,  \eqref{relation11111} follows
$$
  \|f\|_\psi=\sum_{k=0}^\infty\hat{\psi}_k^{-1}
  \sum_{\ell=1}^{Z(d,k)}\hat{f}_{k,\ell}^2=
 \sum_{k=0}^\infty\hat{\phi}_k^{-1}
   \sum_{\ell=1}^{Z(d,k)}\frac{\hat{\phi}_k}{\hat{\psi}_k} \hat{f}_{k,\ell}^2=\|L_{\phi,\psi}^{1/2}f\|_\phi.
$$
This together with  \eqref{operator-relation-112} proves \eqref{norm-relation-123}. 
For any  $f\in\mathcal N_\varphi$ with $\varphi$ satisfying 
\eqref{kernel-relation}  for  $\alpha\geq\beta$ and $0\leq \beta<1$, we obtain
$$
  \|f\|_\varphi=\sum_{k=0}^\infty\hat{\varphi}_k^{-1}
  \sum_{\ell=1}^{Z(d,k)}\hat{f}_{k,\ell}^2=
 \sum_{k=0}^\infty\hat{\psi}_k^{-1}
   \sum_{\ell=1}^{Z(d,k)}\frac{\hat{\psi}_k}{\hat{\varphi}_k} \hat{f}_{k,\ell}^2
   =\sum_{k=0}^\infty\hat{\psi}_k^{-1}
   \sum_{\ell=1}^{Z(d,k)}\hat{\phi}_k^{\beta-\alpha} \hat{f}_{k,\ell}^2.
$$
Setting $\eta(t)=t^{\frac{\beta-\alpha}{2(1-\beta)}}$ and $h'=\mathcal L_{\phi,\psi}^{\frac{\beta-\alpha}{2(1-\beta)}}f$, we have
$
    \|f\|_\varphi=\|h'\|_\psi.
$
The same approach as above and  \eqref{operator-relation-112} verify \eqref{source-condition} and complete  the proof Lemma \ref{Lemma:source-condition}. 
\end{proof}

We then prove Proposition \ref{Prop:operator-sampling} as follows.

\begin{proof}[Proof of Proposition \ref{Prop:operator-sampling}]
For any $\lambda>0$ and diagnose matrix $W$ with positive elements, it follows from $A^{-1}-B^{-1}=B^{-1}(A-B)A^{-1}$ for positive operator $A,B$ that
\begin{eqnarray*}
    &&( L_\phi+\lambda I)^{1/2} (L_{\phi,\Lambda,W}+\lambda I)^{-1}( L_\phi+\lambda I)^{1/2}\\
    &=&
    ( L_\phi+\lambda I)^{1/2} ((L_{\phi,\Lambda,W}+\lambda I)^{-1}-(L_{\phi}+\lambda I)^{-1})( L_\phi+\lambda I)^{1/2}+I\\
    &=&
    ( L_\phi+\lambda I)^{-1/2} (L_{\phi}-L_{\phi,\Lambda,W}) ( L_{\phi,\Lambda,W}+\lambda I)^{-1}
    ( L_\phi+\lambda I)^{1/2}+I.
\end{eqnarray*}
Then 
\begin{eqnarray*}
      \mathcal Q_{\Lambda,\lambda,W}^2  
     &=&
     \|( L_\phi+\lambda I)^{1/2} (L_{\phi,\Lambda,W}+\lambda I)^{-1}( L_\phi+\lambda I)^{1/2}\|_{\phi\rightarrow\phi} \nonumber\\
     &\leq&
     1+\|(L_\phi+\lambda I)^{-1/2}(L_\phi-L_{\phi,\Lambda,W})(L_\phi+\lambda I)^{-1/2}\|_{\phi\rightarrow\phi}\mathcal Q_{\Lambda,\lambda,W}^2.\nonumber 
\end{eqnarray*}
Recalling
$  
     \mathcal R_{\Lambda,\lambda,W,1/2,1/2}  \leq \tilde{c}<1/2,
$  
we get \eqref{def:Q} directly.  To derive \eqref{def:Q*}, we have from \eqref{def:Q} that 
\begin{align*}
    &\|(L_{\phi,D,W}+\lambda I)^{-1/2}  (L_{\phi,D,W}- L_{\phi})(L_{\phi,D,W}+\lambda I)^{-1/2}\|\\
    \leq&
    \|(L_{\phi,D,W}+\lambda I)^{-1/2}(L_{\phi}+\lambda I)^{1/2}\|^2\mathcal R_{\Lambda,\lambda,W,1/2,1/2}  
    \leq 
      \frac{\tilde{c}}{1-\tilde{c}} .
\end{align*}
Therefore, 
\begin{eqnarray*}
    &&(\mathcal Q_{\Lambda,\lambda,W}^*)^2=\|(L_{\phi,D,W}+\lambda I)^{1/2}(L_{\phi}+\lambda
         I)^{-1} (L_{\phi,D,W}+\lambda I)^{1/2}\|_{\phi\rightarrow\phi}\\ 
        &\leq &
  1+\|(L_{\phi,D,W}+\lambda I)^{1/2}((L_{\phi}+\lambda
         I)^{-1}-(L_{\phi,D,W}+\lambda I)^{-1})   (L_{\phi,D,W}+\lambda I)^{1/2}\|_{\phi\rightarrow\phi}\\
         &\leq&
   1+ \frac{\tilde{c}}{1-\tilde{c}} (\mathcal Q_{\Lambda,\lambda,W}^*)^2,
\end{eqnarray*}
which proves \eqref{def:Q*}.
Due to \eqref{def:Q}, \eqref{kernel-relation} and
the  Cordes inequality 
 \cite[Lemma VII.5.5]{bhatia2013matrix} 
\begin{equation}\label{Cordes}
     \|A^s B^s\|_{\mathcal H\rightarrow\mathcal H}\leq \|AB\|^s_{\mathcal H\rightarrow\mathcal H},\qquad 0\leq s\leq 1
\end{equation}
for positive operators $A,B$ defined on $\mathcal H$, we have
\begin{eqnarray*} 
    \|f\|_\psi&=& \|L_\phi^{\frac{1-\beta}{2}}f\|_\phi
    \leq \|L_\phi^{\frac{1-\beta}{2}} (L_{\phi,\Lambda,W}+\lambda I)^{-\frac{1-\beta}{2}} \|_{\phi\rightarrow\phi}\|(L_{\phi,\Lambda,W}+\lambda I)^{\frac{1-\beta}{2}}f\|_\phi  \nonumber\\ 
    &\leq &
    \mathcal Q_{\Lambda,\lambda,W}^{1-\beta} \lambda^{-\beta/2}\|(L_{\phi,\Lambda,W}^{\frac{1}{2}}+\lambda^{1/2})f\|_\phi 
     \leq 
   \mathcal Q_{\Lambda,\lambda,W}^{1-\beta} \lambda^{-\beta/2}(\|f\|_{\Lambda,W}+\lambda^{1/2}\|f\|_\phi)\\
   &\leq&
    \frac{1}{(1-\tilde{c})^{(1-\beta)/2}} \lambda^{-\beta/2}(\|f\|_{\Lambda,W}+\lambda^{1/2}\|f\|_\phi),
\end{eqnarray*}
which proves \eqref{sampling-inequality-sobolev-1.1}.
Similarly, we have from \eqref{def:Q*} that 
\begin{eqnarray*} 
    \|f\|_{\Lambda,W} &=& \|L_{\phi,\Lambda,W}^{\frac1{2}}f\|_\phi
    \leq \|L_{\phi,\Lambda,W}^{\frac{1}{2}} (L_{\phi}+\lambda I)^{-\frac{1}{2}} \|_{\phi\rightarrow\phi}\|(L_{\phi}+\lambda I)^{\frac{1}{2}}f\|_\phi  \nonumber\\ 
    &\leq &
    \mathcal Q_{\Lambda,\lambda,W}^*  \|(L_{\phi}^{\frac{1}{2}}+\lambda^{1/2})f\|_\phi 
     \leq 
   \mathcal Q^*_{\Lambda,\lambda,W} (\|f\|_{L^2}+\lambda^{1/2}\|f\|_\phi)\\
   &\leq&
     \sqrt{\frac{1-\tilde{c}}{1-2\tilde{c}}}(\|f\|_{L^2}+\lambda^{1/2}\|f\|_\phi).
\end{eqnarray*}
 This proves \eqref{sampling-inequality-sobolev-2.1} and completes the proof of Proposition \ref{Prop:operator-sampling}.
\end{proof}

To prove Proposition \ref{Proposition:value-difference-deterministic}, we need the sub-exponential random variable in the following definition \cite[Definition 2.7.5]{vershynin2018high}.

\begin{definition}\label{Def:sub-exponential}
    A mean-zero random variable $\epsilon$ is called a sub-exponential random variable, if  there is a $K'\geq0$ (called as the width) such that 
\begin{equation}\label{moment-for-sub-expontial}
    E\{\exp(u \epsilon)\}\leq \exp((K')^2u^2), \qquad \forall |u|\leq 1/K'. 
\end{equation}
The sub-exponential norm of $\epsilon$ is defined by
\begin{equation}\label{sub-exponential-norm}
     \|\epsilon\|_{SE}:=\inf\{t>0:E\{\exp(|\epsilon|/t)\}\leq 2\}.
\end{equation}
\end{definition}

It can be found in \cite[Page 31]{vershynin2018high} that any sub-Gaussian distribution is clearly sub-exponential. Furthermore, the following lemma derived in \cite[Lemma 2.7.7]{vershynin2018high} shows that the product of sub-Gaussian random variables is a sub-exponential random variable.
\begin{lemma}\label{Lemma:sub-Gaussian-product}
    Let $\varepsilon,\varepsilon'$ be sub-Gaussian random variable. Then $\varepsilon\varepsilon'$ is sub-exponential. Moreover
$$
     \|\varepsilon\varepsilon'\|_{SE}\leq\|\varepsilon\|_{SG}\|\varepsilon'\|_{SG}.
$$
\end{lemma}
We then introduce the following Bernstein's inequality for sub-exponential random variables presented in \cite[Theorem 2.8.2]{vershynin2018high}.
\begin{lemma}\label{Lemma:Bernstei-inequaltiy-random}
    Let $n\in\mathbb N$, $\varepsilon_1,\dots,\varepsilon_{n}$ be independent, mean-zero sub-exponential random variables, and let ${\bf a}=(a_1,\dots,a_n)^T\in\mathbb R^{n}$. Then for every $t\geq 0$, we have 
\begin{equation}\label{Bernstein-for-sub-e}
    P\left\{\left|\sum_{i=1}^{n}a_i\varepsilon_i\right|\geq t\right\}
    \leq 2\exp\left(-c_1'\min\left(\frac{t^2}{(\max_i\|\varepsilon_i\|_{SE})^2\|{\bf a}\|_2^2},\frac{t}{(\max_i\|\varepsilon_i\|_{SE})\|{\bf a}\|_\infty}     \right)\right),
\end{equation}
where $c_1'$ is an absolute constant and  $\|{\bf a}\|_p^p:=\sum_{i=1}^{n}|a_i|^p$.
\end{lemma}
We also need the following lemma  derived in \cite[P.17]{Feng2021radial}.
\begin{lemma}\label{Lemma:effective-dimension}
For any $\lambda>0$,  there holds
 \begin{eqnarray}\label{effective-111}
    \langle ( L_\phi+\lambda I)^{-1}\phi_{x},\phi_{x'}\rangle_\phi
     =
   \sum_{k=0}^\infty(\hat{\phi}_k+\lambda  )^{-1}\hat{\phi}_k\frac{Z(d,k)}{\Omega_{d}}P_k^{d+1}(x\cdot
                  x') .
\end{eqnarray}   
If in addition 
$\hat{\phi}_k\sim k^{-2\gamma}$ with $\gamma>d/2$, then
\begin{eqnarray}\label{effective-estimate}
  \sum_{k=0}^\infty(\hat{\phi}_k+\lambda I)^{-1}\hat{\phi}_k\frac{Z(d,k)}{\Omega_{d}}P_{k}^{d+1}
(x\cdot x)  \leq c_2' \lambda^{-\frac{d}{2\gamma}},
\end{eqnarray}
where $c_2'$ is a constant depending only on $\gamma$ and $d$.
\end{lemma}

With these helps, we  prove Proposition \ref{Proposition:value-difference-deterministic} as follows.

\begin{proof}[Proof of Proposition \ref{Proposition:value-difference-deterministic}]
Since
\begin{eqnarray*} 
   &&\left\|(  L_\phi+\lambda I)^{-1/2}(\mathcal L_{\phi ,\Lambda,W}f^*-S_{\Lambda,W}^T{\bf y}_{\Lambda,W})\right\|_\phi^2\\
   &=&
  \left\langle \sum_{i=1}^{|\Lambda|}w_{i}(f^*(x_i)-y_i)(  L_\phi+\lambda I)^{-1}\phi_{x_i},
   \sum_{i'=1}^{|\Lambda|}w_{i'}(f^*(x_{i'})-y_{i'})\phi_{x_{i'}}\right\rangle_\phi \nonumber\\
   &=&
\sum_{i=1}^{|\Lambda|}\sum_{i'=1}^{|\Lambda|}w_{i}w_{i'}(f^*(x_i)-y_i^*)(f^*(x_{i'})-y_{i'}^*)\langle ( L_\phi+\lambda I)^{-1}\phi_{x_i},\phi_{x_{i'}}\rangle_\phi,   
\end{eqnarray*}
it follows from Lemma \ref{Lemma:effective-dimension}  that
\begin{eqnarray}\label{error-sample-dec-deter}
   &&\left\|(  L_\phi+\lambda I)^{-1/2}(\mathcal L_{\phi ,\Lambda,W}f^*-S_{\Lambda,W}^T{\bf y}_{\Lambda,W})\right\|_\phi^2 \nonumber\\
   &=&
   \sum_{k=0}^\infty\frac{\hat{\phi}_k }{ \hat{\phi}_k+\lambda }\frac{Z(d,k)}{\Omega_{d}}\sum_{i=1}^{|\Lambda|}\sum_{i'=1}^{|\Lambda|}w_{i}w_{i'}\varepsilon_i\varepsilon_{i'}P_k^{d+1}(x_i\cdot x_{i'})\nonumber\\
 &=&
  \sum_{i=1}^{|\Lambda|}\sum_{i'\neq i}\sum_{k=0}^\infty\frac{\hat{\phi}_k }{ \hat{\phi}_k+\lambda }\frac{Z(d,k)}{\Omega_{d}}w_{i}w_{i'}\varepsilon_i\varepsilon_{i'}P_k^{d+1}(x_i\cdot x_{i'})\nonumber \\
&+&
\sum_{i=1}^{|\Lambda|}\sum_{k=0}^\infty\frac{\hat{\phi}_k }{ \hat{\phi}_k+\lambda }\frac{Z(d,k)}{\Omega_{d}}w_{i}^2\varepsilon_i^2P_k^{d+1}(x_i\cdot x_i)\nonumber\\
&=:&
 \tilde{\mathcal I}_1+ \tilde{\mathcal I}_2.  \end{eqnarray}
We at first bound $\tilde{\mathcal I}_2$. Due to the well known Chernoff's inequality 
$$
       P\left[ \xi \geq t\right]  \leq  \frac{  E[e^{\nu\xi}]}{e^{\nu t}},\qquad \forall t,\nu>0  
$$
for an arbitrary random variable $\xi$ and
\begin{equation}\label{Sample-meida-operator}
     \sum_{k=0}^\infty\frac{\hat{\phi}_k }{ \hat{\phi}_k+\lambda }\frac{Z(d,k)}{\Omega_{d}}w_{i}^2P_k^{d+1}(x_i\cdot x_i)
    \leq
   c_2' \lambda^{-\frac{d}{2\gamma}}w_{i}^2
\end{equation}  
which was derived by \eqref{effective-estimate}, 
setting
$$
\xi=\nu\sum_{i=1}^{|\Lambda|}\sum_{k=0}^\infty\frac{\hat{\phi}_k }{ \hat{\phi}_k+\lambda }\frac{Z(d,k)}{\Omega_{d}}w_{i}^2\varepsilon_i^2P_k^{d+1}(x_i\cdot x_i) 
$$
with
$\nu=\frac1{c_2'}\lambda^{\frac{d}{2\gamma}} \left(\sum_{i=1}^{|\Lambda|}w_{i}^{2}\right)^{-1}$,
we  obtain 
$$
\nu\sum_{i=1}^{|\Lambda|}\sum_{k=0}^\infty\frac{\hat{\phi}_k }{ \hat{\phi}_k+\lambda }\frac{Z(d,k)}{\Omega_{d}}w_{i}^2 P_k^{d+1}(x_i\cdot x_i) 
\leq 1
$$
and consequently from \eqref{moment-for-squre-sub-gau-1} with $C_3=1$ and \eqref{effective-estimate} that
\begin{eqnarray*}
      && P\{\tilde{\mathcal I}_2\geq t\}
        \leq   \frac{  E\left\{e^{\nu\sum_{i=1}^{|\Lambda|}\sum_{k=0}^\infty\frac{\hat{\phi}_k }{ \hat{\phi}_k+\lambda }\frac{Z(d,k)}{\Omega_{d}}w_{i}^2\varepsilon_i^2 P_{k}^{d+1}(x_i\cdot x_i)}\right\}}{e^{\nu t}}\\
   & =& 
   \frac{\prod_{i=1}^{|\Lambda|}E\left\{e^{\nu\sum_{k=0}^\infty\frac{\hat{\phi}_k }{ \hat{\phi}_k+\lambda }\frac{Z(d,k)}{\Omega_{d}}w_{i}^2P_{k}^{d+1}(x_i\cdot x_i)\varepsilon_i^2} \right\}}{e^{\nu t}}\\
    &\leq& \frac{\prod_{i=1}^{|\Lambda|}e^{\nu\sum_{k=0}^\infty\frac{\hat{\phi}_k }{ \hat{\phi}_k+\lambda }\frac{Z(d,k)}{\Omega_{d}}w_{i}^2P_{k}^{d+1}(x_i\cdot x_i)}}{e^{\nu t}}\\
    &\leq&
e^{1-   \frac{t}{c_2'\sum_{i=1}^{|\Lambda|}w_{i}^2}\lambda^{\frac{d}{2\gamma}}  }.
\end{eqnarray*}
Therefore, with confidence $1-\delta$ there holds
 \begin{equation}\label{bound-I-2}
 \tilde{\mathcal I}_2\leq 2c_2'\lambda^{-\frac{d}{2\gamma}} \sum_{i=1}^{|\Lambda|} w_{i}^2\log\frac3\delta.
 \end{equation}
Since $\varepsilon_i$ is a sub-Gaussian random variable with sub-Gaussian norm $M$, it follows from Lemma \ref{Lemma:sub-Gaussian-product} that $\varepsilon_i\varepsilon_{i'}$ for fixed $i,i'$ is a sub-exponential random variable with sub-exponential norm not larger than $M^2$.  Furthermore, for any $i\neq i'$, we have $E\{\varepsilon_i\varepsilon_{i'}\}=E\{\varepsilon_i\}E\{\varepsilon_{i'}\}=0$.
Let 
\begin{eqnarray*}
   a_{i,i'}=  \sum_{k=0}^\infty\frac{\hat{\phi}_k }{ \hat{\phi}_k+\lambda }\frac{Z(d,k)}{\Omega_{d}} w_{i}w_{i'} P_k^{d+1}(x_i\cdot
                  x_{i'}),\qquad i,i'=1,\dots, |\Lambda|.
\end{eqnarray*}
We then have from \eqref{effective-estimate} and 
 $P_k^{d+1}(t)\leq 1$ for any $t\in[-1,1]$ that
$$
\max_{i,i'}|a|\leq c_2' \lambda^{-\frac{d}{2\gamma}}w_{\max}^2,\qquad\mbox{and} \quad  \sum_{i,i'=1}^{|\Lambda|}|a_{i,i'}|^2 \leq (c_2')^2 \lambda^{-\frac{d}{\gamma}}\left(\sum_{i=1}^{|\Lambda|}w_{i}^2\right)^2.
$$
Then, Lemma
  \ref{Lemma:Bernstei-inequaltiy-random} shows
$$
    P\left\{\tilde{\mathcal I}_1\geq t\right\}
    \leq 2\exp\left(-c_1'\min\left(\frac{t^2\lambda^{\frac{d}{\gamma}}}{ (c_2')^2 M^4(\sum_{i=1}^{|\Lambda|}w_i^2)^2},\frac{t  \lambda^{\frac{d}{2\gamma}}}{M^2c_2'w_{\max}^2 }   \right)\right).
$$
Hence, with confidence $1-\delta$, there holds
\begin{equation}\label{bound-I-1}
    \tilde{\mathcal I}_1
    \leq
\frac{c_2'}{c_1'}M^2\lambda^{-\frac{d}{2\gamma}}\left(\sum_{i=1}^{|\Lambda|}w_i^2+ w_{\max}^2\right)  \log\frac3\delta.
\end{equation}
Plugging \eqref{bound-I-1} and \eqref{bound-I-2} into \eqref{error-sample-dec-deter}, we have that
$$
    \left\|(  L_\phi+\lambda I)^{-1/2}(\mathcal L_{\phi ,\Lambda,W}f^*-S_{\Lambda,W}^T{\bf y}_{\Lambda,W})\right\|_\phi^2\leq C_6 \lambda^{-d/2\gamma} \left(\sum_{i=1}^{|\Lambda|}w_i^2+ w_{\max}^2\right)
    \log\frac{6}{\delta}
$$
holds with confidence $1-\delta$ for $C_6:= \sqrt{c_2'\left(2+ \frac{M}{c_1'}\right)}$.
This completes the proof of Proposition \ref{Proposition:value-difference-deterministic}.
 \end{proof}

We then prove Proposition \ref{prop:equivalance-11} as follows.
\begin{proof}[Proof of Proposition \ref{prop:equivalance-11}]
The proof of \eqref{equivalance-1} can be found in \cite{lin2023dis}. It suffices to prove \eqref{equivalance-2}. To this end,
we get from
  \eqref{Reproducing-property} that
\begin{eqnarray*}
  &&\mathcal S_{\Lambda,\lambda,W,u}
 =
  \sup_{\|f\|_{\varphi}\leq 1}\sup_{\|g\|_\phi\leq 1}
  \langle (  L_\phi+\lambda I)^{-u} (\mathcal  L_{\phi,\Lambda,W} -J_{\phi}^T)f,g\rangle_\phi\\
  &=&
 \sup_{\|f\|_{\varphi}\leq 1}\sup_{\|g\|_\phi\leq 1}
  \left|\left\langle\int_{\mathbb S^d} f(x')\phi_{x'}d\omega(x')
  -
 \sum_{i=1}^{|\Lambda|}w_{i} f(x_i)\phi_{x_i},  (  L_\phi+\lambda I)^{-u} g\right\rangle_\phi\right|\\
  &=&
  \sup_{\|f\|_{\varphi}\leq 1}\sup_{\|g\|_\phi\leq 1}\left|\int_{\mathbb S^d}f(x')\langle\phi_{x'},  (  L_\phi+\lambda I)^{-u} g\rangle_\phi d\omega(x')-\sum_{i=1}^{|\Lambda|}w_{i}f(x_i)\langle\phi_{x_i},   (  L_\phi+\lambda I)^{-u} g\rangle_\phi\right|\\
  &=&
  \sup_{\|f\|_{\varphi}\leq 1}\sup_{\|g\|_\phi\leq 1}
  \left|\int_{\mathbb S^d}f(x')(  L_\phi+\lambda I)^{-u}g(x')d\omega(x')-\sum_{i=1}^{|\Lambda|}w_{i}f(x_i)(  L_\phi+\lambda I)^{-u} g(x_i)\right|\\
  &=&
{WCE}^{\varphi,\phi}_{\Lambda,\lambda,W,u}.
\end{eqnarray*}
This completes the proof of Proposition \ref{prop:equivalance-11}
 \end{proof}

The proof of Proposition \ref{proposition:quadrature-for-convolution} is motivated by the proof of  \cite[Proposition 4.5]{lin2023dis}, which is based on the following two lemmas  established in \cite[Theorem 2.1]{dai2006generalized} and \cite[Corollary 5.4]{narcowich2007direct}, respectively.
 
\begin{lemma}\label{Lemma:1}
Let $\mathcal Q_{\Lambda,s}:=\{(w_{i,s},  x_i): w_{i,s}\geq 0
\hbox{~and~}   x_i\in \Lambda\}$  be  a positive
 quadrature rule   on $\mathbb S^d$ with degree $s\in\mathbb N$.  For   any  $P\in \mathcal P_{s'}^d$ with $s'\in\mathbb N$, there holds
\begin{eqnarray*}
   \sum_{x_i\in\Lambda}w_{i,s} |P(x_i)|^2  &\leq& c_6'(s'/s)^{d} \|P\|^2_{L^2},\qquad s'>s,\\
\end{eqnarray*}
where $c_6'$ is a constant depending only on $d$.
\end{lemma}

\begin{lemma}\label{Lemma:polynomial-interpolation}
Let $\hat{\phi}_k\stackrel{d}\sim k^{-2\gamma}$ with $\gamma>d/2$.  There exists a constant $c_7'$ depending only on $\gamma$ and $d$, such that for any $f \in\mathcal N_\phi$ and $s^*\geq c_7'/q_{\Lambda}$, there is a $P_{s^*}\in\mathcal P_{s^*}^d$ satisfying $\|P_{s^*}\|_\phi\leq 6\|f\|_\phi$, $\|P_{s^*}\|_{L^2}\leq 6\|f\|_{L^2}$,
\begin{equation}\label{best-app-interpolation1}
     f(x_i)=P_{s^*}(x_i), \qquad i=1,\dots,|\Lambda|,
\end{equation}
and
\begin{equation}\label{best-app-interpolation2}
    \|f-P_{s^*}\|_{L^2}\leq c_8'h_{\Lambda}^{\gamma} \|f\|_\phi,
\end{equation}
where $c_8'$ is a constant depending only on $d$ and $\gamma$.
\end{lemma}

 

We then present the proof of    Proposition \ref{proposition:quadrature-for-convolution} as follows.

\begin{proof}[Proof of Proposition \ref{proposition:quadrature-for-convolution}]
 
  Let $P_f^*$ and $P_g^*$ be, respectively, the corresponding interpolating polynomials in $\mathcal P^d_{s^*}$ 
for $ f\in\mathcal N_\varphi$ and $\eta_{ \lambda,u}( L_\phi)g\in\mathcal N_\phi$ in Lemma \ref{Lemma:polynomial-interpolation} with $s^*= c_7'/q_{\Lambda}$. Then   $\|P_f^*\|_\varphi\leq 6\|  f\|_\varphi $, $\|P_g^*\|_\phi\leq 6\|\eta_{ \lambda,u}( L_\phi)g\|_\phi$,
\begin{equation}\label{inter-poly-1}
     f(x_i)=P_f^*(x_i),\quad \|f-P_f^*\|_{L^2} \leq c_8'h_{\Lambda}^{\alpha \gamma} \|f\|_\varphi,  
\end{equation}
and
\begin{equation}\label{inter-poly-2}
     \eta_{\lambda,u}( L_\phi)g(x_i)=P_g^*(x_i), \quad   
      \| \eta_{\lambda,u}( L_\phi)g-P_g^*\|_{L^2}\leq c_8' h_{\Lambda}^{\gamma} \|\eta_{\lambda,u}( L_\phi)g\|_\phi.
\end{equation}
Therefore, it follows from H\"{o}lder's inequality, \eqref{inter-poly-1} and \eqref{inter-poly-2}   
that 
\begin{eqnarray*} 
    \left|\int_{\mathbb{S}^{d}} f  (x)\eta_{\lambda,u}(L_\phi)g   (x) d \omega(x)-\sum_{x_i\in\Lambda} w_{i, s} \left[f(x_{i}^* )\eta_{\lambda,u}(L_\phi)g  (x_i )\right]\right|  
     \leq 
  \mathcal A_1+\mathcal A_2+\mathcal A_3,  
\end{eqnarray*}
where
\begin{eqnarray*}
  \mathcal A_1&:=&\| f\|_{L^2}\|\eta_{\lambda,u}(L_\phi)g-P_g^*\|_{L^2},\\
  \mathcal A_2&=&
  \|P_g^*\|_{L^2}\|f-P_f^*\|_{L^2},\\
  \mathcal A_3&:=&
  \left|\int_{\mathbb{S}^{d}}  P_f^*(x)P_g^*(x) d \omega(x)-\sum_{x_i\in\Lambda} w_{i, s}P_f^*(x_i)P_g^*(x_i)\right|.
\end{eqnarray*}
Resorting to \eqref{inter-poly-1} and \eqref{inter-poly-2} again and noting $\|f\|_{L^2}\leq \|f\|_\varphi$, we derive  
\begin{eqnarray}
   \mathcal A_1
   &\leq& c_8'h_{\Lambda}^{\gamma} \|\eta_{\lambda,u}(L_\phi)g\|_\phi\|f\|_\varphi\leq c_8'\lambda^{-u}h_{\Lambda}^{\gamma} \|f\|_\varphi\|  g\|_\phi,  \label{bound-A1}\\
   \mathcal A_2 &\leq& 6c_8' h_{\Lambda}^{\alpha \gamma} \|f\|_\varphi\|g\|_\phi. \label{bound-A2}
\end{eqnarray}
It remains to bound $\mathcal A_3$. Let 
 $Q^*_1,Q_2^*\in\mathcal P_{[s/2]}^d$  be the projection of $P_f^*$ and $P_g^*$ onto $\mathcal P_{[s/2]}^d$, respectively, i.e. $\|Q_1\|_{L^2}\leq \|P_f^*\|_{L^2}$, $\|Q_2\|_{L^2}\leq \|P_g^*\|_{L^2}$, and
\begin{equation}\label{best.app-for-p}
     Q_1^*=\arg\min_{P\in \mathcal P_{[s/2]}^d}\|P_f^*-P\|_{L^2},\qquad
     Q_2^*=\arg\min_{P\in \mathcal P_{[s/2]}^d}\|P_g^*-P\|_{L^2},
\end{equation}
where $s$ is the degree of quadrature rule.
The well known Jackson inequality \cite{dai2006jackson} together with the fact $\|P_f^*\|_\varphi\leq 6\|f\|_\varphi $ and $\|P_g^*\|_\phi\leq 6\|\eta_{\lambda,u} (L_{\phi})g\|_\phi$ then yields
\begin{equation}\label{best-app-for-p-rate}
   \|P_f^*-Q_1\|_{L^2}\leq c_9's^{-\alpha \gamma}\|f\|_\varphi,\qquad \|P_g^*-Q_2\|_{L^2}\leq c_9's^{-\gamma}\| \eta_{\lambda,u}( L_\phi)g\|_\phi,
\end{equation}
where $c_9'$ is a constant depending only on $\gamma,\alpha,d$. Form H\"{o}lder's inequality again, we have
$
   \mathcal A_3\leq \sum_{k=1}^5\mathcal A_{3,k},
$
with
\begin{eqnarray*}
  \mathcal A_{3,1}&:=& \|P^*_f-Q_1\|_{L^2}\|P^*_g\|_{L^2},  \\
  \mathcal A_{3,2} &:=&   \|P^*_g-Q_2\|_{L^2}\|Q_1\|_{L^2},   \\
  \mathcal A_{3,3} &:=& \left(\sum_{x_i\in\Lambda}w_{i,s}(P^*_f(x_i)-Q_1(x_i))^2\right)^{1/2}\left(\sum_{x_i\in\Lambda}w_{i,s}(P^*_g(x_i))^2\right)^{1/2},\\
  \mathcal A_{3,4} &:=& \left(\sum_{x_i\in\Lambda }w_{i,s}(P^*_g(x_i)-Q_2(x_i))^2\right)^{1/2}\left(\sum_{x_i\in\Lambda}w_{i,s}(Q_1(x_i))^2\right)^{1/2},\\
   \mathcal A_{3,5} &:=&  \left|\int_{\mathbb{S}^{d}}  Q_1(x)Q_2(x) d \omega(x)-\sum_{x_i\in\Lambda} w_{i, s}Q_1(x_i)Q_2(x_i)\right|.
\end{eqnarray*}
Since $Q_1Q_2\in\mathcal P_s^d$, it is obvious that $\mathcal A_{3,5}=0$. Furthermore, 
 \eqref{best-app-for-p-rate} together with 
$\|Q_1\|_{L^2}\leq \|P_f^*\|_{L^2}$ and $\|Q_2\|_{L^2}\leq \|P_g^*\|_{L^2}$ yields
$$
   \mathcal A_{3,1}\leq 6c_9's^{-\alpha \gamma}\| f\|_\varphi\|\eta_{\lambda,u}(L_\phi)g\|_{L^2}\leq 6c_9'  s^{-\alpha\gamma}
   \|f\|_\varphi\|g\|_\phi 
$$
and
$$
   \mathcal A_{3,2}\leq  6c_9'\lambda^{-u}s^{-\gamma}\|f\|_\varphi\|g\|_\phi.
$$
According to \cref{Lemma:1} with $s'=s^*=c_7'q_{\Lambda}^{-1}$, we  have 
$$
   A_{3,3}\leq c_6'(c_7'q_\Lambda^{-1}s^{-1})^d \mathcal A_{3,1}\leq 6c_1'(c_2'q_\Lambda^{-1}s^{-1})^dc_9' s^{-\alpha \gamma}
   \|f\|_\varphi\|g\|_\phi,
$$
and
$$
   A_{3,4}\leq c_6'(c_7'q_\Lambda^{-1}s^{-1})^d \mathcal A_{3,2}\leq 6c_6'(c_7'q_\Lambda^{-1}s^{-1})^dc_9'\lambda^{-u}s^{-\gamma}
   \|f\|_\varphi\|g\|_\phi.
$$
Combining  the above five estimates, we have
\begin{equation}\label{Bound-A3}
   \mathcal A_{3}\leq
    6(c_6'(c_7'q_\Lambda^{-1}s^{-1})^d +2)
    c_9'(\lambda^{-1/2}s^{-\gamma}+s^{-\alpha\gamma})
   \|f\|_\varphi\|g\|_\phi. 
\end{equation}
We then have from  \eqref{bound-A1}, \eqref{bound-A2} and \eqref{Bound-A3}   that
\begin{eqnarray*}
  &&\left|\int_{\mathbb{S}^{d}}  f(x) \eta_{\lambda,u}(L_{\phi})g)  (x)d \omega(x)-\sum_{x_i\in\Lambda} w_{i, s} \left[f(x_{i}^* )\eta_{\lambda,u}( L_{\phi}+\lambda I) g(x_i )\right]\right|\\
  &\leq&
c_{10}'(1+q_\Lambda^{-1}s^{-1})^d(\lambda^{-1/2}s^{-\gamma}+s^{-\alpha\gamma})
   \|f\|_\varphi\|g\|_\phi,
\end{eqnarray*}
where $c_{10}'$ is a constant depending only on $\tau$, $\gamma,\alpha,u$ and  $d$. This proves \eqref{product-quadruare-222}. The proof of  \eqref{product-quadruare-111} is the same as above, we remove the details for the sake of brevity.
This completes the proof of  Proposition \ref{proposition:quadrature-for-convolution}.
\end{proof}

\subsection{Proofs of results in Section \ref{sec.spectral}}
 
To prove Proposition \ref{Theorem:operator-representation}, we need   two lemmas. The first one can be found in Lemma 7 in \cite{dicker2017kernel} for $u\geq 1$ and Theorem X. 1.1. in \cite{bhatia2013matrix} for $0<u\leq 1$.

\begin{lemma}\label{Lemma:cordes-11}
Let $u\geq 0$ and $A, B$ be positive operators on $\mathcal H$.  
 If  $\|A\|_{\mathcal H\rightarrow\mathcal H},\|B\|_{\mathcal H\rightarrow\mathcal H} \leq F$ for some $F\geq 0$, then
\begin{equation}\label{lip}
     \left\|A^u-B^u\right\|_{\mathcal H\rightarrow\mathcal H}\leq \max\{1,2u F^{u-1}\}\|A-B\|^{\min\{1,u\}}_{\mathcal H\rightarrow\mathcal H}.
\end{equation}
\end{lemma}

The second one focuses on spectral properties for $g_\lambda$.

\begin{lemma}\label{Lemma:media-1}
 Let $g_\lambda$ satisfy \eqref{condition1} and \eqref{condition2}.  
 For any $f\in\mathcal N_\phi$ and $u\in[0,1/2]$, there holds
\begin{equation}\label{media-11}
     \|J_{\phi} g_{\lambda}(L_{\phi,\Lambda,W})(L_\phi+\lambda I)^{u}f\|_{L^2}
  \leq  \sqrt{2}b \mathcal Q_{\Lambda,\lambda,W}^{1+u}\lambda^{u-1/2}\|f\|_\phi. 
\end{equation}
If $f^*\in\mathcal N_\varphi$ with $\varphi,\phi$ satisfying \eqref{kernel-relation}, then for any $\alpha'\in\mathbb R$, there holds
 \begin{equation}\label{media-12}
     \|(L_\phi+\lambda I)^{\alpha'}g_{\lambda}(L_{\phi})J_{\phi}^Tf^*\|_{\phi} 
  \leq  \sqrt{2}b\max\{\kappa^{\frac{\alpha+2\alpha'-1}{4}},1\} )\lambda^{\min\{\frac{\alpha+2\alpha'-1}2,0\}}\|f^*\|_\varphi.
\end{equation}
\end{lemma}

\begin{proof}
 Since $f\in\mathcal N_\phi$, it follows from \eqref{Cordes},  \eqref{def:Q} and \eqref{condition1} that 
\begin{eqnarray*}
   && \|J_{\phi} g_{\lambda}(L_{\phi,\Lambda,W})(L_\phi+\lambda I)^{u}f\|_{L^2}^2\\
   & =&
    \langle J_{\phi} g_{\lambda}(L_{\phi,\Lambda,W})(L_\phi+\lambda I)^{u}f,J_{\phi} g_{\lambda}(L_{\phi,\Lambda,W})(L_\phi+\lambda I)^{u}f\rangle_{L^2}  \\
    &=&
     \langle J_{\phi} ^TJ_{\phi} g_{\lambda}(L_{\phi,\Lambda,W})(L_\phi+\lambda I)^{u}f,g_{\lambda}(L_{\phi,\Lambda,W})(L_\phi+\lambda I)^{u}f\rangle_\phi\\
     &=&
     \langle L_\phi^{1/2} g_{\lambda}(L_{\phi,\Lambda,W})(L_\phi+\lambda I)^{u}f,L_\phi^{1/2} g_{\lambda}(L_{\phi,\Lambda,W})(L_\phi+\lambda I)^{u}f\rangle_\phi \\
     &\leq&
     \| (L_\phi+\lambda I)^{1/2}(L_{\phi,\Lambda,W}+\lambda I)^{-1/2}\|_{\phi\rightarrow\phi}^2\\
     &\times&
     \| (L_{\phi,\Lambda,W}+\lambda I)^{1/2} g_{\lambda}(L_{\phi,\Lambda,W})(L_{\phi,\Lambda,W}+\lambda I)^{u}\|^2_{\phi\rightarrow\phi}\\
     &\times&
     \|(L_{\phi,\Lambda,W}+\lambda I)^{-u}(L_\phi+\lambda I)^{u}\|_{\phi\rightarrow\phi}^2\|f\|_\phi^2\\
     &\leq&
      \mathcal Q_{\Lambda,\lambda,W}^{2+2u} \| (L_{\phi,\Lambda,W}+\lambda I)^{1/2} g_{\lambda}(L_{\phi,\Lambda,W})(L_{\phi,\Lambda,W}+\lambda I)^{u}\|^2_{\phi\rightarrow\phi} \|f\|_\phi^2\\
     &\leq&
      \mathcal Q_{\Lambda,\lambda,W}^{2+2u}\lambda^{2u-1} \|(L_{\phi,\Lambda,W}+\lambda I) g_{\lambda}(L_{\phi,\Lambda,W})\|_{\phi\rightarrow\phi}^2 \|f\|^2_\phi\\
     &\leq&
    2b^2  \mathcal Q_{\Lambda,\lambda,W}^{2+2u}\lambda^{2u-1}\|f\|^2_\phi,
\end{eqnarray*}
which verifies \eqref{media-11}. Similarly, due to \eqref{exchange-adjoint},
we obtain from  \eqref{Cordes}, \eqref{condition1} and Lemma \ref{Lemma:source-condition} that
\begin{eqnarray*}
   && \|(L_\phi+\lambda I)^{\alpha'}g_{\lambda}(L_{\phi})J^T_{\phi} f^*\|_{\phi}^2\\
   & =&
    \langle (L_\phi+\lambda I)^{\alpha'} g_{\lambda}(L_{\phi})J_{\phi}^Tf^*, (L_\phi+\lambda I)^{\alpha'} g_{\lambda}(L_{\phi})J_{\phi}^Tf^*\rangle_\phi \\
    &=&
     \langle J_{\phi} g_{\lambda}(L_{\phi})(L_\phi+\lambda I)^{2\alpha'} g_{\lambda}(L_{\phi})J_{\phi}^Tf^* , f^* \rangle_{L^2}\\
     &=&
     \langle g_{\lambda}(\mathcal L_{\phi})(\mathcal L_\phi+\lambda I)^{2\alpha'} g_{\lambda}(\mathcal L_{\phi})\mathcal L_\phi \mathcal L_\phi^{\frac{\alpha}{2}}h^* , \mathcal L_\phi^{\frac{\alpha}{2}}h^* \rangle_{L^2}\\
      &\leq&
      \|\mathcal L_{\phi}^{\frac{\alpha}{2}} g_{\lambda}(\mathcal L_{\phi})(\mathcal L_{\phi}+\lambda I)^{2\alpha'} g_{\lambda}(\mathcal L_{\phi})\mathcal L_{\phi}\mathcal L_{\phi}^{\frac{\alpha}{2}}\|_{0\rightarrow0}\|h^*\|_{L^2}^2\\
      &\leq&
     \|g_{\lambda}(\mathcal L_{\phi})\mathcal L_{\phi}   \|_{0\rightarrow0} \|(\mathcal L_{\phi}+\lambda I)^{\alpha+2\alpha'} g_{\lambda}(\mathcal L_{\phi}) \|_{0\rightarrow0}\|h^*\|_{L^2}^2\\
      &\leq&
      2b^2\|h^*\|_{L^2}^2\left\{\begin{array}{cc}
         \kappa^{\alpha+2\alpha'-1}  &  \mbox{if}\ \alpha+2\alpha'-1\geq 0,\\
         \lambda^{\alpha+2\alpha'-1}  & \mbox{if}\ \alpha+2\alpha'-1<0.
      \end{array}\right.
\end{eqnarray*}
This verifies \eqref{media-12} by noting $\|h^*\|_{L^2}=\|f^*\|_\varphi$  and 
 completes the proof of  Lemma \ref{Lemma:media-1}.
\end{proof}

\begin{proof}[Proof of Proposition \ref{Theorem:operator-representation}]
Define   
\begin{equation}\label{population-version-1}
      f^\diamond_{D,\lambda,W}=     g_{\lambda}(L_{\phi,\Lambda,W}) \mathcal L_{       \phi,\Lambda,W} f^*. 
\end{equation}
The triangle inequality yields
\begin{equation}\label{Error-dec-11}
    \|J_{\phi}f^\diamond_{D,W,\lambda}-f^*\|_{L^2}
    \leq  
     {\|J_{\phi}f^\diamond_{D,\lambda,W}-f^*\|_{L^2}} 
    +
     \|J_{\phi}(f_{D,W_s,\lambda}-f^\diamond_{D,\lambda,W})\|_{L^2} .
\end{equation} 
We at first bound the second term  by noting 
\begin{eqnarray*}
      &&
      J_{\phi}(f_{D,\lambda,W}-f^\diamond_{D,\lambda,W})
      =
    J_{\phi} g_\lambda(L_{\phi,\Lambda,W}) S^T_{D,W} {\bf y}_{D,W}- J_{\phi} g_{\lambda}(L_{\phi,\Lambda,W}) \mathcal L_{\phi,\Lambda,W} f^*.
\end{eqnarray*}
Since $L_\phi=J_\phi^TJ_\phi$, there holds
\begin{eqnarray*}
      \|J_{\phi}(  L_\phi+\lambda I)^{-1/2}f\|_ {L^2(\mathbb S^d)}^2
    &=&\langle J_{\phi}(  L_\phi+\lambda I)^{-1/2},J_{\phi}(  L_\phi+\lambda I)^{-1/2}\rangle_{L^2}\\
    &=&
    \langle L_\phi^{1/2}(  L_\phi+\lambda I)^{-1/2},L_\phi^{1/2}(  L_\phi+\lambda I)^{-1/2}\rangle_{L^2}, 
\end{eqnarray*}
which consequently follows
\begin{equation}\label{bound-operator-1111111}
    \|J_{\phi}(  L_\phi+\lambda I)^{-1/2}\|_{\phi\rightarrow0}
   \leq 1.
\end{equation} 
Therefore,  \eqref{def:P}, \eqref{def:Q} and \eqref{condition2} yield
\begin{eqnarray}\label{stability-error-operator-represe}
     &&\|J_{\phi}(f_{D,\lambda,W}-f^\diamond_{D,\lambda,W})\|_{L^2}
     \leq  \|J_{\phi}(  L_\phi+\lambda I)^{-1/2}\|_{\phi\rightarrow0} \nonumber\\
     &\times&\|(  L_\phi+\lambda I)^{1/2}g_{\lambda}(L_{\phi,\Lambda,W})(  L_\phi+\lambda I)^{1/2}\|_{\phi\rightarrow\phi} \nonumber\\
     &\times&\|(  L_\phi+\lambda I)^{-1/2}(S^T_{D,W} {\bf y}_{D,W}-\mathcal L_{\phi,\Lambda,W} f^*)\|_\phi \nonumber\\
     &\leq& 
     2b \mathcal Q_{\Lambda,\lambda,W}^2\mathcal P_{D,\lambda,W,0}.
\end{eqnarray}     
We then aim to bound the  first term in  \eqref{Error-dec-11}. Due to \eqref{population-version-1}, triangle inequality  yields
\begin{eqnarray*} 
    &&\|J_{\phi}f^\diamond_{D,\lambda,W}-f^*\|_{L^2}
    \leq 
  \overbrace{\|J_{\phi}g_\lambda(L_\phi)J_{\phi}^Tf^*-f^*\|_{L^2}}^{\mathcal B_1}\\ 
    &+&
    \overbrace{\|J_{\phi} [I-g_{\lambda}(L_{\phi,\Lambda,W})L_{\phi,\Lambda,W}] g_\lambda(L_\phi)J_{\phi}^Tf^*\|_{L^2}}^{\mathcal B_2}\nonumber\\
    &+&
      \overbrace{\|J_{\phi} g_{\lambda}(L_{\phi,\Lambda,W})[ L_{\phi,\Lambda,W} g_\lambda(L_\phi)J_{\phi}^Tf^*-L_\phi g_\lambda(L_\phi)J_{\phi}^Tf^*]\|_{L^2}}^{\mathcal B_3}\nonumber\\
      &+&
   \overbrace{\|J_{\phi} g_{\lambda}(L_{\phi,\Lambda,W})[J_{\phi}^Tf^*-L_\phi g_\lambda(L_\phi)J_{\phi,\psi}^Tf^*]\|_{L^2}}^{\mathcal B_4}\\
    &+&
   \overbrace{\|J_{\phi} g_{\lambda}(L_{\phi,\Lambda,W})(\mathcal  L_{\phi,\Lambda,W} f^*-J_{\phi}^Tf^*)\|_{L^2}}^{\mathcal B_5}.
   \end{eqnarray*}
We then estimate $\mathcal B_k$ for $k=1,\dots,5$, respectively. For $\mathcal B_1$, 
we get from  $\mathcal L_{\phi}=J_{\phi} J_{\phi}^T$,  \cref{condition2} and Lemma \ref{Lemma:source-condition}  that
\begin{eqnarray*} 
   \mathcal B_1
  &=&\|  ( g_{\lambda}(\mathcal L_{\phi})\mathcal L_{\phi}-I) \mathcal L_{\phi}^{\frac{\alpha}{2}}h^*\|_{L^2}
  \leq 
  \|  ( g_{\lambda}(\mathcal L_{\phi})\mathcal L_{\phi}-I) \mathcal L_{\phi}^{\frac{\alpha}{2}}\|_{0\rightarrow0}
  \|h^*\|_{L^2}\\
  &\leq&
   \tilde{C}_0 \lambda^{\min\{\frac{\alpha }{2},\nu_g\}}\|f^*\|_{\varphi}.
\end{eqnarray*}
  For $\mathcal B_2$, since for any $f\in\mathcal N_\phi$,  
\begin{eqnarray}\label{abc-123}
    &&\|J_{\phi} [I-g_{\lambda}(L_{\phi,\Lambda,W})L_{\phi,\Lambda,W}] f\|^2_{L^2} \nonumber\\
   & =&
    \langle L_\phi[I-g_{\lambda}(L_{\phi,\Lambda,W})L_{\phi,\Lambda,W}]f,I-g_{\lambda}(L_{\phi,\Lambda,W})L_{\phi,\Lambda,W}f\rangle_\phi \nonumber\\
    &=&
    \|  L^{1/2}_\phi[I-g_{\lambda}(L_{\phi,\Lambda,W})L_{\phi,\Lambda,W}]f\|^2_\phi,
\end{eqnarray}
we get from \eqref{def:Q} that
\begin{eqnarray*}
      &&\|J_{\phi} [I-g_{\lambda}(L_{\phi,\Lambda,W})L_{\phi,\Lambda,W}] f\|_{L^2}\\
  &\leq &
   \mathcal Q_{\Lambda,\lambda,W}\|  (L_{\phi,\Lambda,W}+\lambda I)^{1/2}[I-g_{\lambda}(L_{\phi,\Lambda,W})L_{\phi,\Lambda,W}]f\|_\phi.
\end{eqnarray*}
Since $0\leq \alpha\leq 1$, we get from the above estimate, \eqref{def:Q*}, \eqref{media-12} with $\alpha'=-\frac{\alpha- 1}{2}$, \eqref{condition2} and \eqref{Cordes} that
\begin{eqnarray*}
    \mathcal B_2
    &\leq &
     \mathcal Q_{\Lambda,\lambda,W}\|  (L_{\phi,\Lambda,W}+\lambda I)^{1/2}[I-g_{\lambda}(L_{\phi,\Lambda,W})L_{\phi,\Lambda,W}](L_{\phi,\Lambda,W}+\lambda I)^{\frac{\alpha- 1}{2}} \|_{\phi\rightarrow\phi}\\
    &\times&\|(L_{\phi,\Lambda,W}+\lambda I)^{-\frac{\alpha- 1}{2}}(L_{\phi}+\lambda I)^{\frac{\alpha -1}{2}}\|_{\phi\rightarrow\phi}\|(L_{\phi}+\lambda I)^{-\frac{\alpha- 1}{2}}g_{\lambda}(L_{\phi})J_{\phi }^Tf^*\|_\phi\\
    &\leq&
    \sqrt{2}b\tilde{C}_0 \|f^*\|_\varphi\mathcal Q_{\Lambda,\lambda,W}( \mathcal Q_{\Lambda,\lambda,W}^*)^{1-\alpha }
    \lambda^{\min\{\frac{\alpha }{2},\nu_g\}}. 
\end{eqnarray*}
For $\mathcal B_3$, we get from \eqref{media-11} with $u=1/2$ and 
$$
    f=(L_\phi+\lambda I)^{-1/2}[ L_{\phi,\Lambda,W} g_\lambda(L_\phi)J_{\phi}^Tf^*-L_\phi g_\lambda(L_\phi)J_{\phi}^Tf^*],
$$
\eqref{Def.RD}, \eqref{media-12} with $\alpha'=1/2$  that 
\begin{eqnarray*}
     \mathcal B_3
     &\leq&
    \sqrt{2}b \mathcal Q_{\Lambda,\lambda,W}^{3/2}\| (L_\phi+\lambda I)^{-1/2} 
    [ L_{\phi,\Lambda,W} g_\lambda(L_\phi)J_{\phi,\psi}^Tf^*-L_\phi g_\lambda(L_\phi)J_{\phi}^Tf^*]\|_\phi\\
    &\leq&
    \sqrt{2}b \mathcal Q_{\Lambda,\lambda,W}^{3/2}
    \| (L_\phi+\lambda I)^{-1/2}(L_{\phi,\Lambda,W}-L_\phi)(L_\phi+\lambda I)^{-1/2} \|_{\phi\rightarrow\phi}
    \|
 (L_\phi+\lambda I)^{1/2}g_\lambda(L_\phi)J_{\phi,\psi}^Tf^*\|_\phi\\
 &\leq&
 2b^2 \max\{\kappa^{\frac{\alpha+2\alpha'-1}{4}},1\}\mathcal Q_{\Lambda,\lambda,W}^{3/2}\mathcal R_{\Lambda,\lambda,W,1/2,1/2}  \|f^*\|_\varphi.
\end{eqnarray*}
For $\mathcal B_4$, 
  we have from Lemma \ref{Lemma:source-condition},
  \eqref{exchange-adjoint} and \eqref{condition2} that 
\begin{eqnarray*}
    &&\|J_{\phi}^Tf^*-L_\phi g_\lambda(L_\phi)J_{\phi}^Tf^*\|_\phi^2
    =
    \langle [I-L_\phi g_\lambda(L_\phi)]J_{\phi}^Tf^*,
    [I-L_\phi g_\lambda(L_\phi)]J_{\phi}^Tf^*\rangle_\phi\\
    &=&
     \langle J_{\phi} [I-L_\phi g_\lambda(L_\phi)]^2J_{\phi}^T \mathcal L_{\phi}^{\frac{\alpha }{2}}h^*,
     \mathcal L_{\phi}^{\frac{\alpha }{2}}h^*\rangle_{L^2}\\
     &=&
     \langle \mathcal L_{\phi}^{\frac{\alpha }{2}} J_{\phi} [I-L_\phi g_\lambda(L_\phi)]^2J_{\phi}^T\mathcal L_{\phi }^{\frac{\alpha-\beta}{2}}h^*,
     h^*\rangle_{L^2}\\
     &=&
     \langle (J_{\phi }J_{\phi }^T)^{\frac{\alpha }{2}}  [I-J_{\phi }J_{\phi }^Tg_\lambda(J_{\phi }J_{\phi }^T)]^2 (J_{\phi }J_{\phi }^T)^{\frac{2+\alpha }{2}}h^*,
     h^*\rangle_{L^2}\\
     &\leq&
     \|(J_{\phi }J_{\phi }^T)^{\frac{\alpha }{2}}  [I-J_{\phi }J_{\phi }^Tg_\lambda(J_{\phi }J_{\phi }^T)]^2 (J_{\phi }J_{\phi }^T)^{\frac{2+\alpha }{2}}\|_{0\rightarrow0}
    \| h^*\|_{L^2}^2\\
    &\leq&
     \tilde{C}_0^2\lambda^{\min\{1+\alpha,\nu_g\}} \| h^*\|_{L^2}^2.
\end{eqnarray*}
Plugging the above estimate into \eqref{media-11} with $u=0$ and $f=J_{\phi}^Tf^*-L_\phi g_\lambda(L_\phi)J_{\phi}^Tf^*$, we get
$$
     \mathcal B_4\leq \left(2b^2  \mathcal Q_{\Lambda,\lambda,W}^2\tilde{C}_0^2\lambda^{\alpha} \| f^*\|_\varphi^2  \right)^{1/2}
     =\sqrt{2}b\tilde{C}_0 \mathcal Q_{\Lambda,\lambda,W}\| f^*\|_\varphi\lambda^{\min\{\frac{\alpha}2,\nu_g\}}.
$$
For $\mathcal B_5$,  \eqref{Lemma:media-1} with $u=1/2$ and $f=(L_\phi+\lambda I)^{-1/2}(\mathcal  L_{\phi,\Lambda,W} f^*-J_{\phi}^Tf^*)$ yields
\begin{eqnarray*}
    \mathcal B_5 \leq  
    \sqrt{2}b \mathcal Q_{\Lambda,\lambda,W}^{3/2} \mathcal S_{\Lambda,\lambda,W,1/2}\|f^*\|_{L^2}
    \leq 
    \sqrt{2}b \mathcal Q_{\Lambda,\lambda,W}^{3/2} \mathcal S_{\Lambda,\lambda,W,1/2}\|f^*\|_\varphi.
 \end{eqnarray*}
Combining all above estimates, we obtain
 \begin{eqnarray*} 
    &&\|J_{\phi}f^\diamond_{D,\lambda,W}-f^*\|
    \leq
    \sqrt{2}\tilde{C}_0\lambda^{\min\{\frac{\alpha}{2},\nu_g\}}\|f^*\|_{\varphi}\big(1+b  \mathcal Q_{\Lambda,\lambda,W}
    +  b\mathcal Q_{\Lambda,\lambda,W}(\mathcal Q_{\Lambda,\lambda,W}^*)^{1-\alpha })    \\
   &+&
     2b^2 \max\{\kappa^{\frac{\alpha+2\alpha'-1}{4}},1\}\mathcal Q_{\Lambda,\lambda,W}^{3/2}\mathcal R_{\Lambda,\lambda,W,1/2,1/2}  \|f^*\|_\varphi
      + 
     \sqrt{2}b \mathcal Q_{\Lambda,\lambda,W}^{3/2} \mathcal S_{\Lambda,\lambda,W,1/2}\|f^*\|_\varphi.
\end{eqnarray*}
Since
$ \mathcal R_{\Lambda,\lambda,W,1/2,1/2}\leq \tilde{c}<1/2, 
$
we  plug   the above inequality and \eqref{stability-error-operator-represe} into \eqref{Error-dec-11} and obtain
\begin{eqnarray*}
     \|J_{\phi}f^\diamond_{D,W,\lambda}-f^*\|_{L^2}
     \leq 
      C_1'  \mathcal P_{D,\lambda,W,0}
      +
      C_1'(\lambda^{\min\{\frac{\alpha}{2},\nu_g\}}  
    + 
       \mathcal R_{\Lambda,\lambda,W,1/2,1/2} 
      + 
      \mathcal S_{\Lambda,\lambda,W,1/2})),
\end{eqnarray*}
where 
\begin{eqnarray*}
 &&C_1' :=\|f^*\|_\varphi\max\{\frac{2b}{1-\tilde{c}},\sqrt{2}\tilde{C}_0\big(1+b  (1-\tilde{c})^{-1/2}
    +  b  (1-2\tilde{c})^{-1}),     \\
   &&\sqrt{2}b (1-\tilde{c})^{-1/2}  \tilde{C}_0 \max\{1,(\alpha-1) \kappa^{\alpha-2}\},
    \max\{\kappa^{\frac{\alpha+2\alpha'-1}{4}},1\}(1-\tilde{c})^{-3/4}  (2b^2+\sqrt{2}b).
\end{eqnarray*}  
This verifies \eqref{Approximation-operator-represen}. 
If $\alpha\geq 1$, we have $f^*\in\mathcal N_\phi$, then we get from \eqref{Holder-type-norm} that
\begin{eqnarray*}
    \| f_{D,\lambda,W}-f^*\|_{\psi} 
   &\leq&
   a_{\tilde{c},\beta} \lambda^{-\beta/2} 
   \left(\lambda^{1/2}\|f_{D,\lambda,W}-f^*\|_\phi+ \|f_{D,\lambda,W}-f^*\|_{L^2} \right)\\
   &\leq&
    a_{\tilde{c},\beta}  \lambda^{-\beta/2} \|(L_\phi+\lambda I)^{1/2}(f_{D,\lambda,W}-f^*)\|_\phi.
\end{eqnarray*}
where $a_{\tilde{c},\beta}:=\frac{1}{(1-\tilde{c})^{(1-\beta)/2}}\left(1+\sqrt{\frac{1-\tilde{c}}{1-2\tilde{c}}}\right) $. Define
\begin{equation}\label{def.noise-free0}
    f^{\star}_{D,\lambda,W}
 :=g_\lambda(L_{\Lambda,\lambda,W})L_{\phi,\Lambda,W}f^*.
\end{equation}
We have from  
$ \mathcal R_{\Lambda,\lambda,W,1/2,1/2}\leq \tilde{c}<1/2,$  \eqref{def:P}, \eqref{def:Q} and \eqref{condition1} that
\begin{eqnarray*}
     \|(L_\phi+\lambda I)^{1/2}(f_{D,\lambda,W}-f^{\star}_{D,\lambda,W})\|_\phi
    \leq
    2b\mathcal Q_{\Lambda,\lambda,W}^2
    \mathcal P_{D,\lambda,W,1}
    \leq 2b(1-\tilde{c})^{-1} \mathcal P_{D,\lambda,W,1}.
\end{eqnarray*}
Moreover, it follows from Lemma \ref{Lemma:source-condition} and \eqref{def:Q} that there exists an $h^*\in L^2$ satisfying $\|h^*\|_{L^2}=\|f^*\|_\varphi$ such that
\begin{eqnarray*}
    &&\|(L_\phi+\lambda I)^{1/2}(f^{\star}_{D,\lambda,W}-f^*)\|_\phi=
 \|(L_\phi+\lambda I)^{1/2}(g_\lambda(L_{\phi,\Lambda,W})L_{\phi,\Lambda,W}-I)
  f^* \|_\phi\\
  &\leq&
  \mathcal Q_{\Lambda,\lambda,W} \|(L_{\phi,\Lambda,W}+\lambda I)^{1/2}(g_\lambda(L_{\phi,\Lambda,W})L_{\phi,\Lambda,W}-I)
     L_\phi^{\frac{\alpha-1}{2}}\|_{\phi\rightarrow\phi} \|h^* \|_{L^2}.
\end{eqnarray*}
If $1\leq \alpha\leq 2$, we get from \eqref{def:Q} and \eqref{condition2} that
\begin{eqnarray*}
    &&\|(L_\phi+\lambda I)^{1/2}(f^{\star}_{D,\lambda,W}-f^*)\|_\phi\\
    &\leq&
   \mathcal Q_{\Lambda,\lambda,W}^\alpha\|f^*\|_\varphi
   \|(L_{\phi,\Lambda,W}+\lambda I)^{\alpha/2}(g_\lambda(L_{\phi,\Lambda,W})L_{\phi,\Lambda,W}-I)\|_{\phi\rightarrow\phi}\\
   &\leq&
   2\tilde{C}_0\mathcal Q_{\Lambda,\lambda,W}^\alpha\lambda^{\min\{\frac{\alpha}{2},\nu_g\}}.
\end{eqnarray*}
If $\alpha>2$, Lemma \ref{Lemma:cordes-11} together with \eqref{Def.RD} implies
\begin{eqnarray*}
    &&\|(L_\phi+\lambda I)^{1/2}(f^{\star}_{D,\lambda,W}-f^*)\|_\phi\\
    &\leq&
   \mathcal Q_{\Lambda,\lambda,W}\|f^*\|_\varphi
   \|(L_{\phi,\Lambda,W}+\lambda I)^{\alpha/2}(g_\lambda(L_{\phi,\Lambda,W})L_{\phi,\Lambda,W}-I)\|_{\phi\rightarrow\phi}\\
   &+&
   \mathcal Q_{\Lambda,\lambda,W}\|f^*\|_\varphi
   \|(L_{\phi,\Lambda,W}+\lambda I)^{1/2}(g_\lambda(L_{\phi,\Lambda,W})L_{\phi,\Lambda,W}-I)\|_{\phi\rightarrow\phi}\|L_{\phi,\Lambda,W}^{\frac{\alpha-1}{2}}-L_\phi^{\frac{\alpha-1}{2}}\|_{\phi\rightarrow\phi}\\ 
   &\leq&
   2\tilde{C}_0\mathcal Q_{\Lambda,\lambda,W}
   (\lambda^{\min\{\frac{\alpha}2,\nu_g\}}+\lambda^{\min\{\frac{1}{2},\nu_g\}}\max\{1,(\alpha-1) \kappa^{\alpha-2}\}\mathcal R_{\Lambda,\lambda,W,0,0}).
\end{eqnarray*}
Since
$ \mathcal R_{\Lambda,\lambda,W,1/2,1/2}\leq \tilde{c}<1/2,$, we obtain from \eqref{def:Q} that 
\begin{eqnarray}\label{Apprxo.1234}
    \|(L_\phi+\lambda I)^{1/2}(f^{\star}_{D,\lambda,W}-f^*)\|_\phi
    \leq
 C_2'(\lambda^{\min\{\frac{\alpha}{2},\nu_g\}}+\lambda^{\min\{\frac{1}{2},\nu_g\}}\mathcal R_{\Lambda,\lambda,W,0,0}\mathbb I_{\alpha>2}),
\end{eqnarray}
where $C_2':=2\tilde{C}_0 (1-\tilde{c})^{-2}\max\{1,(\alpha-1) \kappa^{\alpha-2}\}$.
All these yields
\begin{eqnarray}\label{error-pre-sobolev}
    \|(L_\phi+\lambda I)^{1/2}(f_{D,\lambda,W}-f^*)\|_\phi&\leq &
     C_3'(\lambda^{\min\{\frac{\alpha}{2},\nu_g\}}+\lambda^{\min\{\frac{1}{2},\nu_g\}}\mathcal R_{\Lambda,\lambda,W,0,0}\mathbb I_{\alpha>2})\nonumber\\
     &+&
     2b(1-\tilde{c})^{-1} \mathcal P_{D,\lambda,W,1}
\end{eqnarray}
Plugging the above estimate into \eqref{Holder-type-norm}, we have
\begin{eqnarray*} 
     \| f_{D,\lambda,W}-f^*\|_{\psi} 
       &\leq& 
       2ba_{\tilde{c},\beta}   (1-\tilde{c})^{-1} \lambda^{-\beta/2}\mathcal P_{D,\lambda,W,1}\\
       &+&
       C_2'a_{\tilde{c},\beta} \lambda^{-\beta/2}(\lambda^{\min\{\frac{\alpha}{2},\nu_g\}}+\lambda^{\min\{\frac{1}{2},\nu_g\}}\mathcal R_{\Lambda,\lambda,W,0,0}\mathbb I_{\alpha>2}).
\end{eqnarray*}
This completes the proof of Theorem \ref{Theorem:operator-representation}. 
\end{proof}

\begin{proof}[Proof of Corollary \ref{Corollary:app-rate-quasi}]

Since $\lambda=|\Lambda|^{-\frac{2\gamma}{2\gamma \alpha +d}}$ and $s=c_\diamond|\Lambda|^{1/d}$, there exists a $C'_3$ satisfying the condition of Proposition \ref{proposition:quadrature-for-convolution} such that 
$
     \lambda\geq C's^{-2\gamma}.  
$
Under this circumstance,   
it follows from Theorem \ref{Theorem:operator-representation} and Corollary \ref{Cor:quasi-uniform-integral}  that  with confidence $1-\delta$, there holds
   $$
   \|J_{\phi}f_{D,W_s,\lambda}-f^*\|_{L^2} 
   \leq  \max\{C_3',\bar{C}_2\}(|\Lambda|^{-1/2}\lambda^{-\frac{d}{4\gamma}}\log\frac6\delta
   +\lambda^{\min\{\frac{\alpha}{2},\nu_g\}}+s^{-\alpha\gamma}).
$$
Since $\nu_g\geq \frac{\alpha}{2}$ and $\alpha\gamma>d/2$, 
direct computation yields \eqref{approxiamtion-error-random}. This
  proves Corollary \ref{Corollary:app-rate-quasi}.
 \end{proof}
\begin{proof}[Proof of Corollary \ref{Corollary:large-optimal}]
 Due to \eqref{Approximation-operator-represen-large} and Corollary \ref{Cor:quasi-uniform-integral}, with confidence $1-\delta$, there holds
\begin{eqnarray*}
    \|f^\diamond_{D,W,\lambda}-f^*\|_{\psi}
      \leq  
       \bar{C}_1\lambda^{-\beta}   (C_9|\Lambda|^{-1/2}\lambda^{-d/(4\gamma)}\log\frac6\delta
      + 
        \lambda^{\alpha/2}+C_9c_\diamond^{-\gamma} \lambda^{1/2}|\Lambda|^{-\gamma/d}).  
\end{eqnarray*}
The same argument as that in the proof of Corollary \ref{Corollary:app-rate-quasi} then yields that with confidence $1-\delta$, there holds
$$
   \|f^\diamond_{D,W,\lambda}-f^*\|_{\psi}
   \leq C_4'|\Lambda|^{- \frac{2(\alpha-\beta)\gamma}{2\alpha\gamma+d}}\log\frac{6}{\delta},
$$
where $C_4'$ is a constant independent of $|\Lambda|,\delta,\lambda$.
This completes the proof of Corollary \ref{Corollary:large-optimal}.
\end{proof}

\subsection{Proofs of results in Section \ref{Sec:parameter}}

\begin{proof}[Proof of Proposition \ref{Proposition:stopping-rule}]
Due to \eqref{def:Q}, we have from the triangle inequality that
\begin{eqnarray*}
     &&\left\|(  L_{\phi,D,W}+\lambda I)^{1/2} ( f_{D,W,\lambda}-f_{D,W,\lambda'})\right\|_\phi\\
     &\leq&
     \mathcal Q_{\Lambda,\lambda,W}
    (
      \|(  L_{\phi}+\lambda I)^{1/2} ( f_{D,W,\lambda}-f^*)\|_\phi+\|(  L_{\phi}+\lambda I)^{1/2} (f_{D,W,\lambda'}-f^*) \|_\phi).
\end{eqnarray*}
Then, it follows from \eqref{error-pre-sobolev}, $\mathcal R_{\Lambda,\lambda,W,1/2,1/2}, \mathcal R_{\Lambda,\lambda',W,1/2,1/2} \leq \tilde{c}<1/2$, $\lambda>\lambda'$  and \eqref{def:Q} that 
\begin{eqnarray*}
     &&\left\|(  L_{\phi,D,W}+\lambda I)^{1/2} ( f_{D,W,\lambda}-f_{D,W,\lambda'})\right\|_\phi\\
     &\leq&
      2(1-\tilde{c})^{-1/2} C_2'(\lambda^{\min\{\frac{\alpha}{2},\nu_g\}}+\lambda^{\min\{\frac{1}{2},\nu_g\}}\mathcal R_{\Lambda,\lambda,W,0,0}\mathbb I_{\alpha>2} )\\
      &+&
      8b(1-\tilde{c})^{-3/2} \mathcal P_{D,\lambda',W,1},
\end{eqnarray*}
where we used the fact that $\mathcal P_{D,\lambda,W,1}$ decreases with respect to $\lambda$.
 This completes the proof of Proposition \ref{Proposition:stopping-rule} by noting $C_2':=2\tilde{C}_0 (1-\tilde{c})^{-2}\max\{1,(\alpha-1) \kappa^{\alpha-2}\}$. 
\end{proof}

We then use Proposition \ref{Proposition:stopping-rule} to prove Theorem \ref{Theorem:Lepskii}.

\begin{proof}[Proof of Theorem \ref{Theorem:Lepskii}]
Due to the definition of $K_q$ and $s=c_\diamond|\Lambda|^{1/d}$, there exists a $k_0\in[1,K_q]$ such that $\lambda_{k_0}=q_0q^{k_0}\stackrel{d}\sim |\Lambda|^{-\frac{2\gamma}{2\alpha\gamma+d}}$. If $k_0\leq\hat{k}$, i.e., $\lambda_{k_0}\geq\lambda_{\hat{k}}$, we obtain from the definition of $\hat{k}$ that
\begin{eqnarray*}
     32c_*^2C_6b(3/2)^{-3/2}\lambda_{\hat{k}+1}^{-d/(4\gamma)} |\Lambda|^{-1/2}\log\frac6\delta
     <
     \left\|( L_{\phi,\Lambda,W_s}+\lambda_{\hat{k}+1} I)^{1/2} ( f_{D,W_s,\lambda_{\hat{k}+1}}-f_{D,W_s,\lambda_{\hat{k}}})\right\|_\phi.
\end{eqnarray*}
But \eqref{stop-1111} together with Corollary  \ref{Cor:quasi-uniform-integral} yields
\begin{eqnarray*}
     &&\left\|(  L_{\phi,\Lambda,W_s}+\lambda_{\hat{k}+1} I)^{1/2} ( f_{D,W_s,\lambda_{\hat{k}+1}}-f_{D,W_s,\lambda_{\hat{k}}})\right\|_\phi \nonumber\\
      &\leq &
      \bar{C}_5 (\lambda_{\hat{k}+1}^{\alpha/2}+C_9c_\diamond^{-\gamma} \lambda_{\hat{k}+1}^{1/2}|\Lambda|^{-\gamma/d} \mathbb I_{\alpha>2} )
      +
      16\bar{C}_5c_*^2C_6b(3/2)^{-3/2} \lambda_{\hat{k}+1}^{-d/(4\gamma)} |\Lambda|^{-1/2}\log\frac6\delta.
\end{eqnarray*}
Hence,
$$
 16\bar{C}_5c_*^2C_6b(3/2)^{-3/2} \lambda_{\hat{k}+1}^{-d/(4\gamma)} |\Lambda|^{-1/2}\log\frac6\delta
 \leq 
 \bar{C}_5 (\lambda_{\hat{k}+1}^{\alpha/2}+C_9c_\diamond^{-\gamma} \lambda_{\hat{k}+1}^{1/2}|\Lambda|^{-\gamma/d} 
 \mathbb I_{\alpha>2} ).
$$
Recalling $\lambda_{k_0}\geq\lambda_{\hat{k}}$ and
$$
  \lambda_{\hat{k}}^{-d/(4\gamma)}\leq 2\lambda_{\hat{k}+1}^{-d/(4\gamma)},
$$
we obtain from \eqref{error-pre-sobolev} and Corollary \ref{Cor:quasi-uniform-integral} that 
\begin{eqnarray*}
    &&\left\|(  L_{\phi,\Lambda,W_s}+\lambda_{\hat{k}} I)^{1/2} ( f_{D,W_s,\lambda_{\hat{k}}}-f^*)\right\|_\phi\\
    &\leq&
    C_5' (\lambda_{\hat{k}}^{\alpha/2}+ \lambda_{\hat{k}}^{1/2}|\Lambda|^{-\gamma/d} \mathbb I_{\alpha>2} 
      +
    \lambda_{\hat{k}}^{-d/(4\gamma)} |\Lambda|^{-1/2}\log\frac6\delta)\\
      &\leq&
      C_6' 
      |\Lambda|^{-\frac{\alpha\gamma}{2\alpha\gamma+d}},
\end{eqnarray*}
 where $C_5',C_6'$ are constants independent of $\delta,|\lambda|$.
If $k_0>\hat{k}$, i.e., $\lambda_{k_0}<\lambda_{\hat{k}}$,   the definition of $\hat{k}$,  \eqref{def:Q},  \eqref{error-pre-sobolev} and Corollary \ref{Cor:quasi-uniform-integral} yield
\begin{eqnarray*}
    &&\left\|(  L_{\phi,\Lambda,W_s}+\lambda_{\hat{k}} I)^{1/2} ( f_{D,W_s,\lambda_{\hat{k}}}-f^*)\right\|_\phi\\
    &\leq&
    \left\|(  L_{\phi,\Lambda,W_s}+\lambda_{\hat{k}} I)^{1/2} ( f_{D,W_s,\lambda_{\hat{k}}}-f_{D,W_s,\lambda_{k_0}})\right\|_\phi
    +
     \left\|(  L_{\phi,\Lambda,W_s}+\lambda_{\hat{k}} I)^{1/2} (  f_{D,W_s,\lambda_{k_0}}-f^*)\right\|_\phi\\
     &\leq&
     \sum_{k=\hat{k}+1}^{k_0}\left\|(  L_{\phi,\Lambda,W_s}+\lambda_{k+1} I)^{1/2} ( f_{D,W_s,\lambda_{\hat{k}+1}}-f_{D,W_s,\lambda_{\hat{k}}})\right\|_\phi+
    \sqrt{\frac23} \left\|(  L_{\phi}+\lambda_{\hat{k}} I)^{1/2} (  f_{D,W_s,\lambda_{k_0}}-f^*)\right\|_\phi\\
    &\leq&
     \sum_{k=\hat{k}+1}^{k_0}C_{LP}\lambda_k^{-d/(4\gamma)} |\Lambda|^{-1/2}\log\frac6\delta
     +
     \sqrt{\frac23} \left\|(  L_{\phi}+\lambda_{k_0} I)^{1/2} (  f_{D,W_s,\lambda_{k_0}}-f^*)\right\|_\phi\\
     &\leq&
    C_7' \left( |\Lambda|^{-1/2} \sum_{k=1}^{k_0}q^{-kd/(4\gamma)}+|\Lambda|^{-\frac{\alpha\gamma}{2\alpha\gamma+d}}\right)\log\frac6\delta
    \leq
    C_{8}'|\Lambda|^{-\frac{\alpha\gamma}{2\alpha\gamma+d}}\log\frac6\delta,
\end{eqnarray*}
where $C_7',C_{8}'$ are constants independent of $|\Lambda|,\delta$.
Combining the above two cases, i.e., $\hat{k}\geq k_0$ and $\hat{k}<k_0$, we obtain 
$$
   \left\|(  L_{\phi,\Lambda,W_s}+\lambda_{\hat{k}} I)^{1/2} ( f_{D,W_s,\lambda_{\hat{k}}}-f^*)\right\|_\phi\leq  C_{9}'|\Lambda|^{-\frac{\alpha\gamma}{2\alpha\gamma+d}}\log\frac6\delta,
$$
where $C_{9}'=\max\{C_{6}',C_8'\}$. Plugging the above estimate into \eqref{Holder-type-norm} and noting $\tilde{c}=1/3$, we then have 
\begin{eqnarray*}
   \|f^\diamond_{D,W_s,\lambda_{\hat{k}}}-f^*\|_{\psi}
     \leq   \bar{C}_6 |\Lambda|^{- \frac{2(\alpha-\beta)\gamma}{2\alpha\gamma+d}}\log\frac{6}{\delta},
\end{eqnarray*}
where $\bar{C}_6$ is a constant independent of $|\Lambda|,\delta$, This completes the proof of Theorem \ref{Theorem:Lepskii}.
\end{proof}

 \subsection{Proofs of results in Section \ref{sec:Dis}} 
To prove Theorem \ref{Theorem:DS-optimal}, we need the following lemma   provided in \cite{Feng2021radial}.

\begin{lemma}\label{Lemma:Error-decomposition}
Let $\overline{f}_{D,W_{\vec{s}},\vec{\lambda}}$ be defined by (\ref{global estimator}). If $\{\varepsilon_{i,j}\}_{i=1,j=1}^{|\Lambda_j|,J}$ is a set of mean-zero random variables, we have
\begin{eqnarray*} 
        &&  E[\|\overline{f}_{D,W_{\vec{s}},\vec{\lambda}}-f^*\|_{\psi}^2] \nonumber\\
        & \leq&
         \sum_{j=1}^J \frac{|\Lambda_j|^2}{|\Lambda|^2}  E\left[\| f_{D_j,W_{j,s_j},\lambda_j}
         -f^*\|_{\psi}^2\right]
        +\sum_{j=1}^J\frac{|\Lambda_j|}{|\Lambda|}
        \left\|f^\star_{D_j,W_{j,s_j},\lambda_j}-f^*\right\|_{\psi}^2,
\end{eqnarray*}
where $f^{\star}_{D_j,W_{j,s_j},\lambda_j}$ is defined similarly as \eqref{def.noise-free0} by
\begin{equation}\label{population-version-3}
      f^{\star}_{D_j,W_{j,s_j},\lambda_j}=  g_\lambda(L_{\Lambda_j,\lambda_j,W_{s_j}})L_{\phi,\Lambda_j,W_{s_j}}f^*. 
\end{equation}
\end{lemma}
  
We then use Theorem \ref{Theorem:operator-representation}, Corollary \ref{Cor:quasi-uniform-integral}, Corollary \ref{Corollary:app-rate-quasi} and Lemma \ref{Lemma:Error-decomposition} to prove Theorem \ref{Theorem:DS-optimal}.

\begin{proof}[Proof of Theorem \ref{Theorem:DS-optimal}]
It follows from Theorem \ref{Theorem:operator-representation} that                        \begin{eqnarray*}
         &&\|f_{D_j,W_{j,s_j},\lambda_j}
         -f^*\|_{\psi}^2
         \leq
        \bar{C}_1\lambda^{-\beta/2} ( \mathcal P_{D_j,\lambda_j,W_{s_j},1}
      +
       \lambda_j^{\frac{\alpha}{2}}+\lambda^{\frac{1}{2}}\mathcal R_{\Lambda_j,\lambda_j,W,0,0}\mathbb I_{\alpha>2}). 
\end{eqnarray*}
Furthermore, \eqref{Apprxo.1234} together with \eqref{Holder-type-norm} yields
 \begin{eqnarray*} 
    &&E[ \|f^\star
    _{D_j,W_{s_j},\lambda_j}-f^*\|_\psi^2 ] 
    \leq
   \bar{C}' 
    (\lambda_j^{ \frac{\alpha-\beta}{2}}+\lambda_j^{  \frac{1-\beta}{2}}\mathcal R_{\Lambda_j,\lambda_j,W_{s_j},0,0}\mathbb I_{\alpha>2}),
\end{eqnarray*}
where $\bar{C}'$ is a constant depending only on $C_2'$.
Plugging the above two estimates into \cref{Lemma:Error-decomposition} and noting $s_j\sim |\Lambda_j|^{1/d}$, $\lambda_j\sim  |\Lambda|^{-\frac{2\gamma}{2\gamma(\alpha-\beta)+d}}$,  $|\Lambda_1|\sim\dots\sim|\Lambda_J|$ and Corollary \ref{Cor:quasi-uniform-integral}, we obtain 
\begin{eqnarray*}
     E[\|\overline{f}_{D,W_{\vec{s}},\vec{\lambda}}-f^*\|_{\psi}^2]  \leq 
       \bar{C}_1'\sum_{j=1}^J \frac{|\Lambda_j|}{|\Lambda|} \left( \frac{1}{|\Lambda|} \lambda_j^{-\frac{d}{2\gamma}} 
        +    
   \lambda_j^{\alpha-\beta} \right) 
    \leq
    \tilde{C}_1|\Lambda|^{-\frac{\gamma(\alpha-\beta)}{2\gamma(\alpha-\beta})+d},
\end{eqnarray*}
where $\bar{C}_1'$ and $\tilde{C}_1$ are constants independent of $\lambda_j,|\Lambda_j|, J$. This completes the proof of Theorem \ref{Theorem:DS-optimal}.
\end{proof}

\bibliographystyle{siamplain}
\bibliography{references}

\end{document}